\documentclass[reqno]{amsart}
\usepackage{CJK}
\usepackage{hyperref}
\usepackage{cite}
\usepackage{amsmath}
\usepackage{mathrsfs}
\usepackage{color}

\makeatletter
\@namedef{subjclassname@2020}{%
	\textup{2020} Mathematics Subject Classification}
\makeatother

\numberwithin{equation}{section}

\newtheorem{thm}{Theorem}[section]
\newtheorem{prop}[thm]{Proposition}
\newtheorem{lem}[thm]{Lemma}

\newtheorem{example}{Example}[section]
\newtheorem{Def}[thm]{Definition}

\theoremstyle{definition}

\newtheorem{remark}[thm]{Remark}
\allowdisplaybreaks
\newcommand{\R}{{\mathbb{R}}}

\newcommand{\rn}{{\mathbb{R}^N}}
\newcommand{\N}{\mathbb{N}}

\DeclareMathOperator{\DIV}{div}

\title[Entropy solutions to the fully non-local diffusion equations]{Entropy solutions to the fully non-local diffusion equations}
\author{Ying Li,  Chao Zhang$^*$}
\address{Ying Li\newline
School of Mathematics, Harbin Institute of Technology, Harbin 150001, China\newline
\texttt{lymath@hit.edu.cn}}
\address{Chao Zhang\newline
 School of Mathematics and Institute for Advanced Study in Mathematics, Harbin Institute of Technology, Harbin 150001, China
\newline
\texttt{czhangmath@hit.edu.cn}}
\thanks{$^*$ Corresponding author.}
\thanks{{\bf Acknowledgment}:
 This work was supported by the National Natural Science Foundation of China (No. 12071098).}
\begin{document}
\maketitle
\begin{abstract}
We consider the fully non-local diffusion equations with non-negative $L^1$-data. Based on the approximation and energy methods, we prove the existence and uniqueness of non-negative entropy solutions for such problems. In particular, our results are valid for the time-space fractional Laplacian equations.
\end{abstract}

\section{Introduction}
Let $\Omega$ be a bounded domain of $\rn (N>2)$. The present work is  concerned with the existence of solutions to the following integro-differential  equation
\begin{equation}\label{eq:main}
\left\{\begin{array}{cl}
\partial_t(k*(u-u_0))+(-\Delta)^{s}u=f &\quad \text{in }  \Omega_T \equiv \Omega\times (0,T),\\
u=0 &\quad \text{in } (\rn \setminus \Omega)\times (0,T),\\
u(x,0)=u_0(x) &\quad  \text{in }  \Omega.
\end{array}\right.
\end{equation}
Here, $k\in L^{1}_{loc}(\mathbb{R^+})$ is a non-negative kernel that belongs to a  certain kernel class, and $k*v$ denotes the convolution on the positive half-line w.r.t the time variable, that is $(k*v)(t)=\int^{t}_{0}k(t-\tau)v(\tau)\,d\tau$ with $t\geq 0$. The fractional Laplacian operator $(-\Delta)^{s}$ with $0<s<1$  is defined as
\begin{align*}
    (-\Delta)^{s} u(x,t) :=& \quad {\rm P.V.} \int_{\rn}\frac{u(x,t)-u(y,t)}{|x-y|^{N+2s}}\,dy \\
    =& \quad \lim_{\epsilon\to 0}\int_{\rn\setminus B_\epsilon(x)}\frac{u(x,t)-u(y,t)}{|x-y|^{N+2s}}\,dy,
\end{align*}
where $(x,t)\in \rn \times \mathbb{R^{+}}$ and P.V. is a commonly used abbreviation in the principal value sense.  We consider the above problem with non-negative $L^1$-data, i.e.,
\[ 0\leq f\in L^1(\Omega_T) \quad \text{and} \quad 0\leq u_0\in L^1(\Omega).\]
The kernel $k$ satisfies the following condition:
\begin{itemize}
    \item[(Hk)]  $k$ is non-negative, non-increasing, and there  exists a kernel $l\in L^p(0,T)$ with $p>1$ such that $k*l=1$ in $(0, \infty)$.
\end{itemize}

Let $k_{\lambda}$ be the kernels which arised in characterisation \eqref{eqdefBlambda} of the  Yosida approximation. To be able to show the existence of entropy solutions to \eqref{eq:main}, we further assume that $k$ satisfies the following conditions:
\begin{itemize}
    \item[(K1)] There exist constants $C_1,C_2>0$ such that
    \[0\leq k_{\lambda}(t)\leq C_1k(t)+C_2,~\lambda>0,~t\in (0,T);\]
    \item[(K2)] $k\in {\rm AC}_{loc}((0,T])$ and there exist constants $C_1,C_2>0$ such that
     \[0\leq -k'_{\lambda}(t)\leq -C_1k'(t)+C_2,~\lambda>0,~t\in (0,T),\]
     where $k'$ represents the differentiation of $k$. Moreover, the convergence  $k'_{\lambda}(t)\to k'(t)$ as $\lambda\to 0$  holds  for almost every $t\in(0,T)$.
\end{itemize}

Note that kernels $k$ satisfying  $\rm(Hk)$  are in particular kernels of type $\mathcal{PC}$, see Definition~\ref{defk}. The kernels of type $\mathcal{PC}$ was introduced by Zacher in \cite{Rico2009weak}.  $\mathcal{PC}$-kernels cover most of the relevant integro-differential operators w.r.t. time  that appear in physics applications in the context of subdiffusion processes, which is a special case of anomalous diffusive behaviour.  We refer to \cite{Ralf2000the,Ralf2004the} and references therein for the physical background. There have been wide research activities \cite{jukka2016decay,Samko2003integro,vicente2015optimal,zacher2013aweak,Rico2008boundedness} on the kernels of  $\mathcal{PC}$ type.

Next, we shall give some examples of $\mathcal{PC}$-kernels that satisfy the conditions (Hk), (K1) and (K2). 

\begin{example}
The most typical example is given by
  \[ k(t)=g_{1-\alpha}(t),~ l(t)=g_{\alpha}(t),\quad t>0,\]
where $\alpha\in (0,1)$ and $g_{\beta}$ denotes the Riemann-Liouville kernel
\[g_{\beta}=\frac{t^{\beta-1}}{\Gamma(\beta)},\quad t>0, \quad\beta>0.\]
In this case, the term $\partial_{t}(k*\cdot)$ denotes the  fractional time derivative of order $\alpha$  in the sense of  Riemann-Liouville, and $k*\partial_{t}(\cdot)$ the Caputo derivative if $v$ is sufficient smooth, see  \cite[Example 2.7]{scholtes2018existence}.
\end{example}
\begin{example} Another example is the time-fractional case with exponential weight, i.e.,
\[k(t)=g_{1-\alpha}(t)e^{-\mu t}, \quad l(t)= g_{\alpha}(t)e^{-\mu t}+\mu (1*[g_{\alpha}(\cdot)e^{-\mu\cdot}])(t),\quad t>0,\]
where $\mu >0$ and $\alpha\in (0,1)$. We refer to  \cite[Example 2.8]{scholtes2018existence} for more information.
\end{example}

Our focus in this manuscript is to establish the existence and uniqueness of solutions for problem \eqref{eq:main}. As we consider problem with data of low integrability, it is reasonable to work with entropy solutions, which need less regularity of the data than usual weak solutions. The notion of entropy solutions has been proposed  by B\'{e}nilan et al.\cite{benilan1995an} for the study of  the nonlinear elliptic problems.  There already exists a vast literature that concerns about the entropy solutions for elliptic and parabolic problems with  $L^1$-data, see \cite{boris2008on,lucio1996existence,k2019entropy,li2021entropy,zhang2010entropy,zhang2010renormalized}.

In recent years, the study of  the fractional Laplacian operators and non-local operators have received significant attention. The main reason lies in  its wide range of applications, such as  continuum mechanics, phase transition phenomena, population dynamics, image process, game theory and so on. We refer to \cite{luis2012nonlocal,luis2007an,luis2011uniform,Ralf2004the} and the references therein  for more details. We point out that the linear elliptic problems with non-local operator $(-\Delta)^{s}_{p}$ in the case $p=2$ were  considered in \cite{nathael2010renormalized,Kenneth2011a,Tommaso2015basic}.  The corresponding parabolic problem was also discussed in \cite{Tommaso2015basic}. It is worth mentioning that in \cite{nathael2010renormalized}, Alibaud et al.  established the existence and uniqueness of solution for the problem
\[\beta(u)+(-\Delta)^su  \ni f \quad \text{in}~\mathbb{R}^N,\]
where $f\in L^1(\mathbb{R}^N)$ and $\beta$ is a maximal monotone graph in $\mathbb{R}$. Moreover, there are  a large number of papers devoted to the study of elliptic and parabolic equations with non-local operators $(-\Delta)^{s}_{p}$ in the case  $p\neq 2$. The existence of a unique entropy positive solution to fractional elliptic $p$-Laplacian equation with weight and general datum was developed by Abdellaoui et al. \cite{abdellaoui2019on},  see also  \cite{abdellaoui2018on} for the parabolic framework. For more recent works concern about the non-local operator $(-\Delta)^s_p$, we refer the reader to \cite{Janne2019equiv,Tuomo2015nonlocal,mazon2016frac,Teng2019renormalized}.

As to the problems with integro-differential operators, we mention that Jakubowski and  Wittbold  \cite{Volker2004on} generalized the concept of entropy solutions for parabolic equations with $L^1$-data and considered a class of nonlinear history-dependent degenerated elliptic-parabolic equation. The existence of entropy solutions for the doubly nonlinear history-dependent problems of the form
\[\partial_{t}[k*(b(u)-b(u_0))]- \DIV a(x,\nabla u)=f\]
was obtained by Scholtes and Wittbold in \cite{scholtes2018existence}, where $k$ is assumed to be of type $\mathcal{PC}$, $b$ is strictly increasing and continuous  satisfying $b(0)=0$, and $a$ is a Carath\'{e}dory function  satisfying the Leary-Lions condition. Sapountzoglou \cite{niklas2020entropy} generalized this result  to the case that $b$ is not strictly increasing.  In addition,  \cite{schmitz2023entropy,petra2021bounded} were devoted to the study of  existence of solutions for the time-fractional porous medium type equations, which were equipped with bounded measurable coefficients that may explicitly depend on time. Moreover,  we point out that an abstract evolutionary  integro-differential equation in Hilbert spaces with a kernel $k$ of type $\mathcal{PC}$ was investigated in \cite{Rico2009weak}. Vergara and Zacher \cite{vicente2015optimal} have studied the decay estimates of time-fractional equations via energy methods, see also ~\cite{jukka2016decay} on this issue.

The equation  which is non-local both in space and time is called  the fully non-local diffusion equation.  Fully non-local equations play a crucial role in model situations with long-range interactions and memory effects and have been proposed  to describe plasma transport, see \cite{D1,D2}. The study of fully non-local equations has  attracted considerable interest during recent years. Allen, Caffarelli and Vasseur \cite{Mark2016a} have discussed  the regularity of weak solutions to a parabolic problem with fractional diffusion in space and a fractional time derivative. Kim and Lim \cite{kim2016asy} have considered the behavior of fundamental solutions  to time-space fractional differential equations. The decay estimates for time-space fractional equations were investigated  in \cite{Jukka2017repre,dipierro2019decay}. For other recent developments on the issues of the fully non-local problems, let us refer to \cite{fu2022global,Dier2020on,Corta2021}.

To our best knowledge, there is no result concerns about the fully non-local problems with the right-hand side is merely integrable. Inspired by the papers mentioned above, we aim  to establish the existence and uniqueness results of entropy solution for the fully non-local  equation \eqref{eq:main}. We point out that we cannot use the method of Steklov average in time to obtain the  appropriate time regularization since the Steklov average operators and convolution do not commute. We shall introduce a  regularization in time which adapts to the non-local nature of the problem \eqref{eq:main}, see Definition \ref{defshijian}. The method is a modification of the regularization method proposed by Landes \cite{landes1989on} and has been used in  \cite{scholtes2018existence,niklas2020entropy}. It is important to note  that the authors in  \cite{scholtes2018existence,niklas2020entropy} obtained the existence and uniqueness of entropy solutions for the time-fractional problems by using the theory of accretive operators. They showed that the generalized solution of the associated abstract Volterra equation is an entropy solution. Unlike the proofs of  \cite{scholtes2018existence,niklas2020entropy}, we first construct an approximate problem of \eqref{eq:main} and obtain the existence and uniqueness of weak solutions for the approximate problem by using the results in \cite{Rico2009weak}. Next, we shall establish some a priori estimates for the approximate solution sequence.  Then we draw a subsequence to obtain a limit function and prove that the limit function is the entropy solution for problem \eqref{eq:main}. The uniqueness of the entropy solution is obtained by a comparison  principle.
In addition,  we use the method of Yosida approximation to regularize  the kernel $k$ which will be discussed in Section 2.
This method has already been used in \cite{niklas2020entropy,zacher2013aweak,Rico2009weak,Vergara2008lyapunov}.

Before giving the main results of this manuscript,  we first introduce some functional settings that will be used below. We refer to \cite{Eleonora2012hit,Tommaso2015basic,Teng2019renormalized} for more details.

    Let $s\in (0,1)$ and $p>1$. The fractional Sobolev space
    \[W^{s,p}(\mathbb{R}^N) \equiv   \bigg\{ u\in L^p(\mathbb{R}^N): \int_{\mathbb{R}^N}\int_{\mathbb{R}^N}\frac{|u(x)-u(y)|^p}{|x-y|^{N+ps}}\,dxdy <\infty \bigg\}\]
    is a Banach space endowed with the norm
    \[\|u\|_{W^{s,p}(\mathbb{R}^N)}=\|u\|_{L^p(\mathbb{R}^N)}+\bigg( \int_{\mathbb{R}^N}\int_{\mathbb{R}^N}\frac{|u(x)-u(y)|^p}{|x-y|^{N+ps}}\,dxdy\bigg)^{\frac{1}{p}}.
\]
For $\omega\in W^{s,p}(\mathbb{R}^N)$, we define the fractional $p$-Laplacian as
\[(-\Delta)^{s}_{p}\omega(x)= {\rm P.V.}\int_{\mathbb{R}^N}\frac{|\omega(x)-\omega(y)|^{p-2}(\omega(x)-\omega(y))}{|x-y|^{N+ps}}\,dy.\]
Note that  for all $\omega,v\in W^{s,p}(\mathbb{R}^N)$, we have
\[\langle(-\Delta)^{s}_{p}\omega(x), v\rangle=\frac{1}{2}\int_{\mathbb{R}^N}\int_{\mathbb{R}^N} |\omega(x)-\omega(y)|^{p-2}(\omega(x)-\omega(y)) (v(x)-v(y))\,dv.\]

Denote $\mathcal{D}_{\Omega}=(\mathbb{R}^N\times\mathbb{R}^N)\setminus (
\mathcal{C}\Omega\times \mathcal{C}\Omega)$, where $\mathcal{C}\Omega=\mathbb{R}^N\setminus \Omega$.
$X^{s,p}$ denotes the linear space of Lebesuge measurable function $u:\mathbb{R}^N\to \mathbb{R}$ such that the quantity
\[X^{s,p}(\Omega)=\bigg(  \int_{\Omega}|u|^p\,dx +\int_{\mathcal{D}_{\Omega}}\frac{|u(x)-u(y)|^p}{|x-y|^{N+sp}}\,dxdy
\bigg)\]
is finite. $X_{0}^{s,p}$  denotes the space of functions $u\in X^{s,p} $
that vanish a.e. in $\mathcal{C}\Omega$. As it is  explained in \cite{Eleonora2012hit,Teng2019renormalized},  we know that there exists a positive constant $C$ such that for any $u\in X^{s,p}_{0}(\Omega)$,
\[\int_{\mathcal{D}_{\Omega}}|u(x)-u(y)|^p\,dv \leq \|u\|^p_{W^{s,p}(\mathbb{R}^N)}\leq C\int_{\mathcal{D}_{\Omega}}|u(x)-u(y)|^p\,dv, \]
where
\[dv=\frac{dxdy}{|x-y|^{N+sp}}.\]
Thus, we can endow $X^{s,p}_{0}(\Omega)$ with the equivalent norm
\[\|u\|_{X^{s,p}_{0}(\Omega)}=\bigg(\int_{\mathcal{D}_{\Omega}} |u(x)-u(y)|^p\,dv\bigg)^{\frac{1}{p}}.\]
It is obvious that for all $\omega, v \in X_{0}^{s,p}(\Omega)$,
\[\langle(-\Delta)^{s}_{p}\omega(x), v\rangle=\frac{1}{2}\int_{\mathcal{D}_{\Omega}}|\omega(x)-\omega(y)|^{p-2}(\omega(x)-\omega(y)) (v(x)-v(y))\,dv.\]
Moreover, we notice that $(-\Delta)^{s}_{p}: X^{s,p}_{0}(\Omega)\hookrightarrow (X^{s,p}_{0}(\Omega))' $, where $(X^{s,p}_{0}(\Omega))'$ denotes  the dual space of $X^{s,p}_{0}(\Omega)$.
In addition, we claim that the space $L^p(0,T; X^{s,p}_{0}(\Omega))$ is defined as the set of function $u$ such that $u\in L^p(\Omega_{T})$ with $\|u\|_{L^p(0,T;X^{s,p}_{0}(\Omega))}<\infty$, where
\[\|u\|_{L^p(0,T;X^{s,p}_{0}(\Omega))}=\bigg(\int^{T}_{0}\int_{\mathcal{D}_{\Omega}}|u(x,t)-u(y,t)|^p\,dvdt\bigg)^{\frac{1}{p}}.\]
We define an entropy solution to problem \eqref{eq:main} based on the definition in \cite{Volker2004on}. We set
\[\mathcal{P}:=\{S\in C^1{(\mathbb{R})}:0\leq S' \leq 1, \quad \text{supp}\, S'~\text{compact}, \quad S(0)=0\}.\]
In the following definition and throughout the paper, $T_{K}$ denotes the truncation function at level $K>0$:
\[T_{K}(r):=\max\{ \min(r,K),-K\}.\]
Moreover, we denote
$T_{K,M}(r):=T_{M}(r)-T_K(r
)$ for $M>K>0$.
\begin{Def}
    A measurable function $u\in L^1(\Omega_T)$ with $T_{k}(u)\in L^2(0,T;X^{s,2}_{0}(\Omega))$ is called an entropy solution to \eqref{eq:main}  if
    \begin{align}\label{Defentropysoluiton}
        &-\int^{T}_{0}\int_{\Omega}\zeta_t\bigg[k_1*\int^{u}_{u_0}S(\sigma-\phi)\,d\sigma\bigg]\,dxdt+\int^{T}_{0}\int_{\Omega}\zeta\partial_t[k_2*(u-u_0)]S(u-\phi)\,dxdt\\
        &\quad +\frac{1}{2}\int^{T}_{0}\int_{\mathcal{D}_{\Omega}} (u(x)-u(y))\bigg(S(u(x,t)-\phi(x,t))-S(u(y,t)-\phi(y,t))\bigg)\zeta\,dvdt\nonumber\\
        &\quad\leq \int^{T}_{0}\int_{\Omega}\zeta fS(u-\phi)\,dxdt\nonumber
    \end{align}
  for all $\phi\in X^{s,2}_{0}(\Omega)\cap L^{\infty}(\Omega)$, $\zeta\in C^{\infty}_{0}([0,T))$, $\zeta\geq 0$, $S\in \mathcal{P}$, and $k_1,k_2 \in L^1(0,T)$ non-increasing and non-negative with $k=k_1+k_2$ and $k_2(0^+)<\infty$.
\end{Def}

Our main results read as follows.
\begin{thm}\label{theorem:main}
 Let $0\leq f\in L^1(\Omega_T)$, $0\leq u_0\in L^1(\Omega)$ and let $k$  satisfy the conditions \rm(Hk), \rm(K1) and \rm(K2). Then there exists an entropy  solution for problem \eqref{eq:main}.
\end{thm}

\begin{prop}[ \textbf{Comparison  principle}]\label{prop}
    Suppose that $u^{i}(i=1,2)$ are entropy solutions to problem \eqref{eq:main} with $f=f^{i}\in L^1(\Omega_{T})$ and $u_0=u^{i}_0\in L^1(\Omega)$. Then we have
    \begin{equation}\label{eqprop1}
     \int^{T}_{0}\int_{\Omega}(u^1-u^2)^+\,dxdt \leq T\int_{\Omega} (u_{0}^{1}-u^{2}_{0})^+\,dx+\|l\|_{L^1(0,T)}\int^{T}_{0}\int_{\Omega}(f^1-f^2)^+\,dxdt
     \end{equation}
     and
     \begin{equation}\label{eqprop2}
 \|u^1-u^2\|_{L^1(\Omega_{T})}\leq T\|u_{0}^{1}-u^{2}_{0}\|_{L^{1}(\Omega)}+\|l\|_{L^1(0,T)}\|f^1-f^2\|_{L^1(\Omega_{T})}.
         \end{equation}
\end{prop}
\begin{remark}
Proposition~\ref{prop} provides that the entropy solution of problem \eqref{eq:main} is unique and non-negative.
\end{remark}

\section{Preliminaries}
In this section,  we shall give a short introduction to kernels of type $\mathcal{PC}$. For detailed discussion we refer the reader to \cite{clement1981asymptotic,Vergara2008lyapunov,Rico2009weak}.  We begin with the definition of the $\mathcal{PC}$-kernels.
\begin{Def}\label{defk}
    A kernel $k\in L^1_{loc}([0,\infty))$ is called to be of type $\mathcal{PC}$ if it is non-negative, non-increasing, and there exists a kernel $l\in L^1_{loc}([0,\infty))$ such that
    \begin{equation*}
        (k*l)(t)=1,\quad \forall t>0.
    \end{equation*}
    We write it as $(k,l)\in \mathcal{PC}$.
\end{Def}
We next introduce an important method of approximation kernels of type $\mathcal{PC}$. For $1\leq p<\infty$, $T>0$, and a real Banach space $X$ we consider the operator $B$ defined by
\begin{equation}\label{eqdefB}
   Bu=\partial_{t}(k*u),\quad D(B)=\{u\in L^p(0,T;X):k*u\in {}_{0}W^{1,p}(0,T;X)\},
\end{equation}
where  the zero means vanishing at $t=0$. It is known that this operator is $m$-accretive in $L^p(0,T;X)$, see \cite{clement1990completely,clement1992global}.
Its Yosida approximation $B_{\lambda}$, defined by $B_{\lambda}=\frac{B}{I+\lambda B}$, $\lambda>0$, enjoy the property that for any $u\in D(B)$,
\[ B_{\lambda}u\to Bu \quad \text{in}~L^p(0,T;X)  \text{ as}~\lambda \to 0. \]
Denote  $J^B_{\lambda}=(I+\lambda B)^{-1}$, $\lambda>0$, it had been proved in \cite[Theorem 2.1]{Vergara2008lyapunov} that the Yosida approximation $B_{\lambda}=BJ^{B}_{\lambda}$ could be wirtten in the form
\begin{equation}\label{eqdefBlambda}
    B_{\lambda}u=\frac{B}{I+\lambda B}u=BJ^{B}_{\lambda}u=\partial_{t} (k_{\lambda}*u),
\end{equation}
where $k_{\lambda}=\frac{s_{\lambda}}{\lambda}=k*r_{\lambda}$, and $s_{\lambda},r_{\lambda}\in L^1_{loc}(\R^+)$ are defined by  the scalar Volterra equations
\begin{equation*}\label{eqS}
    \begin{split}
     &s_{\lambda}(t)+\lambda^{-1} (l*s_{\lambda})(t)=1, \qquad t>0, \\
& r_{\lambda}(t)+\lambda^{-1} (l*r_{\lambda})(t)=\lambda^{-1}l(t), \qquad t>0.
\end{split}
\end{equation*}
Moreover, we have
\[J^B_{\lambda}u=\frac{u}{(I+\lambda B)}=r_{\lambda}*u.\]
In addition, according to the formula for the variation of the constants of the Volterra integral equation \cite[Theorem 3.5]{Gripenberg1990volterra}, it follows that
\begin{equation*}
    s_{\lambda}(t)=1-\int^{t}_{0}r_{\lambda}(\tau)d\tau,~ t>0,\lambda>0.
\end{equation*}
Thus, we know that $s_{\lambda}\in W^{1,1}(0,T)$. We point out that both $s_{\lambda}$ and $r_{\lambda}$ are non-negative for all $\lambda>0$. This is a consequence of  $l$ is completely positive as $(k,l)\in\mathcal{PC}$, see \cite[Theorem~2.2]{clement1981asymptotic}.
It is further known that the kernels $k_{\lambda}$ are also non-negative and non-increasing and  belong to $W^{1,1}(0,T)$.  We note
that for any  $v\in L^p(0,T; X)$,
\begin{equation}
    J^{B}_{\lambda}v =r_{\lambda}*v \to v \quad \text{in } L^p(0,T; X)~\text{as}~\lambda\to 0.
\end{equation}
In particular, we have
\begin{equation}
    k_{\lambda}=r_{\lambda}*k \to k\quad \text{in}~L^1(0,T)~\text{as}~\lambda\to 0.
    \end{equation}
In fact, for any $v\in L^p(0,T; X)$, it is easy to check that
\[k*l*v=1*v=\int^{t}_{0}v(\tau)\,d\tau \in {}_{0}W^{1,p}(0,T;X),\]
which implies that
\[l*v\in D(B)=\{\omega\in L^p(0,T; X): k*\omega \in  {}_{0}W^{1,p}(0,T;X) \}.\]
Hence, we deduce that
\[r_{\lambda}*v= \partial_{t}(k*r_{\lambda}*l*v)=\partial_{t}(k_{\lambda}*l*v)=
B_{\lambda}(l*v)
\to B(l*v)=v\]
in $L^p(0,T;X)$ as $\lambda\to 0$.

Next, we will give a  regularization in time which adapts to the nonlocal nature of the problem \eqref{eq:main}. We point out that the method is a modification of the regularzation method proposed by Landes \cite{landes1989on}.
\begin{Def}\label{defshijian}
    Let $X$ be a Banach space, $X'$ its dual and $1\leq p'<\infty$. For $v\in L^{p'}(0,T;X')$, let $v_{\mu}\in L^{p'}(0,T;X')$ be defined by
    \begin{equation}\label{eqmu}
        v_{\mu}(t)=\int^{T}_{t}r_{\mu}(\tau-t)v(\tau)\,d\tau,~t\in (0,T),~\mu>0.
    \end{equation}
\end{Def}
Note that  $v_{\mu}=J^{B^*}_{\mu}v$ for $p'=\frac{p}{p-1}$, where $B^*:D(B^*)\subset L^{p'}(0,T;X')\to L^{p'}(0,T;X')$ is the adjoint of the operator $B$ defined in \eqref{eqdefB}. Thus, we have $v_{\mu}\to v$ in $L^{p'}(0,T;X')$ for any $v\in L^{p'}(0,T;X')$ as $\mu\to 0$.
\begin{lem}\rm \cite[Lemma 2.3]{scholtes2018existence}
    Let $\mu>0$, and assume that $v \in L^\infty(\Omega_T)$. Then $v_{\mu}\in L^{\infty}(\Omega_{T})$ and
    \begin{equation*}
    \|v_{\mu}\|_{L^{\infty}(\Omega_{T})}\leq \|v\|_{L^{\infty}(\Omega_{T})}.
    \end{equation*}
\end{lem}
Then, we give a fundamental identify for integro-differential operator of the form $\partial_{t}(k*u)$.
\begin{lem}\rm\label{lemconvex} Let $k\in W^{1,1}(0,T)$, $H\in C^1(\mathbb{R})$ and $u\in L^1(0,T)$ with $u(t)\in \mathbb{R}$ for almost all $t\in (0,T)$. Suppose that the functions $H(u)$, $H'(u)u$ and $H'(u)(k'*u)$ belong to $L^1(0,T)$. Then we have for almost all $t\in (0, T)$,
\begin{equation}\label{eqlemmaconvev}
\begin{split}
    & H'(u(t))\partial_{t}(k*u)(t)\\
    &\quad=\partial_{t}(k*H(u))(t)+(H'(u(t))u(t)-H(u(t)))k(t)\\
    &\qquad+\int^{t}_{0}\bigg(H(u(t-s))-H(u(t))-H'(u(t))[u(t-s)-u(t)]\bigg)[-k'(s)]\,ds.
\end{split}
\end{equation}
\end{lem}

The assertion of Lemma~\ref{lemconvex} follows from a straightforward computation, see also \cite{zacher2013aweak,Rico2008boundedness}. We note that a more general version of \eqref{eqlemmaconvev} in integrated form can be found in \cite[Lemma 18.4.1]{Gripenberg1990volterra}. The following two lemmas are immediate consequence of \eqref{eqlemmaconvev}.

\begin{lem}\rm\cite[Corollary 6.1]{jukka2016decay}\label{lemH}
    Let   $u_0\in \mathbb{R}$, $k$, $H$ and $u$ be as in Lemma~\rm \ref{lemconvex}. And assume in addition that $k$ is non-negative and non-increasing and $H$ is convex. Then
    \begin{equation}
        H'(u(t)) \partial_{t}(k*(u-u_0))(t)\geq \partial_{t}(k*[H(u)-H(u_0)])(t)~\text{a.a.}~t\in (0,T).
    \end{equation}
\end{lem}

\begin{lem}\rm \cite[Lemma 2.5]{scholtes2018existence}\label{lem:equality}
    Let $u\in L^{1}(0,T)$, $k$ a $\mathcal{PC}$-kernel, and for $\lambda>0$, let $k_{\lambda}$ be the kernel of the Yosida approximation of the operator defined  in \eqref{eqdefBlambda}.  Then we have for all $K>0$ and almost every $t>0$,
    \begin{equation}
        \begin{split}
            &\partial_{t}[k_{\lambda}*u ](t)T_{K}(u(t))\\
        &\quad=\partial_t[k_{\lambda}*\int^{u}_{0}T_K(\sigma)\,d\sigma ](t)+\bigg[T_{K}(u(t))u-\int^{u}_{0}T_{K}(\sigma)\,d\sigma\bigg]k_{\lambda}(t)\\
        &\qquad+\int^{t}_{0}\bigg[\int^{u(t-s)}_{u(t)}T_{K}(\sigma)\,d\sigma-T_{K}(u(t))(u(t-s)-u(t))\bigg][-k'_{\lambda}(s)]\,ds\\
        &\quad\geq \partial_t\left[k_{\lambda}*\int^{u}_{0}T_K(\sigma)\,d\sigma\right](t).
        \end{split}
    \end{equation}
\end{lem}

\section{Weak solutions}
In this section,  we will introduce an approximate problem. For $m\in \N$ we define $f^m=T_{m}(f)$ and $u^{m}_{0}=T_{m}(u_0)$, then we know that $f^m$ and $u^{m}_{0}$ are non-negative, $(f^{m},u^{m}_0) \in L^{\infty}(\Omega_T)\times L^{\infty}(\Omega)$ and $(f^{m},u^{m}_{0}) \nearrow (f,u_0)$ strongly in $L^1(\Omega_T)\times L^1(\Omega)$ such that
\begin{equation}
    \|f^{m}\|_{L^1(\Omega_T)} \leq \|f\|_{L^1(\Omega_T)}, \quad \|u^{m}_{0}\|_{L^1(\Omega)}\leq \|u_0\|_{L^1(\Omega)}.
\end{equation}
We set
\[W(u_0, X^{s,2}_{0}(\Omega),L^2(\Omega)):=\{\omega\in L^{2}(0,T; X_{0}^{s,2}(\Omega)):k*(\omega-u_0)\in {}_{0}W^{1,2}(0,T;(X_{0}^{s,2}(\Omega))')\}.\]
Then we consider the approximate problem of \eqref{eq:main}.
\begin{lem}\label{lemma:appro} Let \rm(Hk) be satisfied. Then the following problem
\begin{equation}\label{eq:mainappro}
\left\{\begin{array}{cl}
\partial_t(k*(u^{m}-u^{m}_0))+(-\Delta)^{s}u^{m}=f^{m} &\quad \text{in }  \Omega_T \equiv \Omega\times (0,T),\\
u^{m}=0 &\quad \text{in } (\rn \setminus \Omega)\times (0,T),\\
u^{m}(x,0)=u^{m}_0(x) &\quad  \text{in }  \Omega
\end{array}\right.
\end{equation}
admits a unique  weak solution $u^m\in W(u_0, X^{s,2}_{0}(\Omega),L^2(\Omega))$ with $k*u^m\in C([0,T];L^2(\Omega))$ such that for any $\phi\in L^2(0,T;X^{s,2}_{0}(\Omega))\cap L^{\infty}(\Omega_T)$, $t_1\in (0,T)$,
\begin{equation}\label{eqweaksolution}
\begin{split}
   &\int^{t_1}_{0}\int_{\Omega} \phi \partial_{t}[k*(u^m-u^m_0)]\,dxdt+\frac{1}{2}\int^{t_1}_{0}\int_{\mathcal{D}_{\Omega}}(u^m(x)-u^m(y))(\phi(x)-\phi(y))\,dvdt\\
   &\quad=\int_{0}^{t_1}\int_{\Omega} f^m\phi\,dxdt
\end{split}
\end{equation}
holds.
\end{lem}
\begin{proof}
Arguing as the arguments of~\cite[Theorem 3.1]{Rico2009weak}, we can find a unique weak solution $u^{m}\in W(u^m_0, X^{s,2}_{0}(\Omega),L^2(\Omega))$ for problem ~\eqref{eq:mainappro}.
\end{proof}

\begin{lem}[\textbf{Comparison principle}]\label{lemmabijiao}
    Suppose that $u^m_{i} (i=1,2)$ are weak solutions to problem \eqref{eq:main} with $u^m_0=u^m_{0i}\in L^{\infty}(\Omega)$ and $f=f^m_{i}\in L^{\infty}(\Omega_T)$. Then, we have
\begin{eqnarray}\label{eqbijiao1}
     \int^{T}_{0}\int_{\Omega}(u^m_1-u^m_2)^{+}\,dxdt \leq T\int_{\Omega}(u^m_{01}-u^m_{02})^{+}\,dx+\|l\|_{L^1(0,T)}\int^{T}_{0}\int_{\Omega}(f^m_1-f^m_2)^{+}\,dx
    \end{eqnarray}
and
 \begin{eqnarray}\label{eqbijiao2}
 \|u^m_1-u^m_2\|_{L^1(\Omega_T)} \leq T \|u^m_{01}-u^m_{02}\|_{L^1(\Omega)}+\|l\|_{L^1(0,T)}\|f^m_1-f^m_2\|_{L^1(\Omega_T)}.
 \end{eqnarray}
\end{lem}
\begin{proof} We shall give the proof of \eqref{eqbijiao1} and the proof of \eqref{eqbijiao2} is similar to \eqref{eqbijiao1}.

Set
 \[v=u^m_1-u^m_2,\quad v_0=u^m_{01}-u^m_{02}\quad and \quad F=f^m_1-f^m_2.\]
For $\epsilon>0$, we define a convex function
     \begin{equation}\label{eqH}
         H_{\epsilon}(y)=\sqrt{y^2+\epsilon^2}-\epsilon,\quad y\in \R.
     \end{equation}
After computation, we can get
\[H'_{\epsilon}(y)=\frac{y}{\sqrt{y^2+\epsilon^2}},\quad H''_{\epsilon}(y)=\frac{\epsilon^2}{(y^2+\epsilon^2)^{\frac{3}{2}}},\quad y\in \R.\]

Taking $H'_{\epsilon}(v^+)\in  L^2(0,T;X^{s,2}_{0}(\Omega))\cap L^{\infty}(\Omega_T)$ as a test function in the weak formulation of the problem for both $u^m_1$ and $u^m_2$, we deduce that
\begin{equation*}
\begin{split}
 &\int^{t_1}_{0}\int_{\Omega}H'_{\epsilon}(v^+) \partial_{t}[k*(v-v_0)]\,dxdt+\frac{1}{2}\int^{t_1}_{0}\int_{\mathcal{D}_{\Omega}}(v(x)-v(y))(H'_{\epsilon}(v^+(x))-H'_{\epsilon}(v^+(y)))\,dvdt\\
   &\quad=\int_{0}^{t_1}\int_{\Omega} FH'_{\epsilon}(v^+)\,dxdt.
\end{split}
\end{equation*}
It is not difficult to see that
\begin{equation*}
\begin{split}
  &\int^{t_1}_{0}\int_{\mathcal{D}_{\Omega}}(v(x)-v(y))(H'_{\epsilon}(v^+(x))-H'_{\epsilon}(v^+(y)))\,dvdt\\
  &\quad=\int^{t_1}_{0}\int_{\mathcal{D}_{\Omega}}(v(x)-v(y))(v^+(x)-v^+(y))H_{\epsilon}''(\xi)\,dvdt \geq 0,
\end{split}
\end{equation*}
where $\xi$ is between $v^+(x)$ and $v^+(y)$.
Combining with the fact that $0<H'_{\epsilon}(v^+)\leq 1$,  we deduce that
\begin{equation}
    \int^{t_1}_{0}\int_{\Omega}H'_{\epsilon}(v^+) \partial_{t}[k*(v-v_0)]\,dxdt  \leq   \int_{0}^{t}\int_{\Omega} FH'_{\epsilon}(v^+)\,dxdt\leq \int^{t_1}_{0}\int_{\Omega}F^+\,dxdt.
\end{equation}
For $\lambda>0$, let $k_{\lambda}$ be the kernel associated to the Yosida approximation of the operator
\begin{eqnarray*}
B_1 \omega&:=&\partial_t(k*\omega)\\
    D(B_1)&:=&\{\omega\in L^2(0,T; (X_{0}^{s,2}(\Omega))'): k*\omega \in {}_{0}W^{1,2}(0,T;(X_{0}^{s,2}(\Omega))')\}.
\end{eqnarray*}
Note that
\begin{equation*}
    \begin{split}
        \int^{t_1}_{0}\int_{\Omega}H'_{\epsilon}(v^+) \partial_{t}[k*(v-v_0)]\,dxdt &=\int^{t_1}_{0}\int_{\Omega}H'_{\epsilon}(v^+) \partial_{t}[k_{\lambda}*(v-v_0)]\,dxdt \\
        &\quad-\int^{t_1}_{0}\int_{\Omega}H'_{\epsilon}(v^+) \partial_{t}[(k_{\lambda}-k)*(v-v_0)]\,dxdt.
    \end{split}
\end{equation*}
 Since $H'_{\epsilon}(v^+)$ is a convex function, it follows from Lemma~\ref{lemH} that
\begin{equation*}
\begin{split}
     \int^{t_1}_{0}\int_{\Omega}H'_{\epsilon}(v^+) \partial_{t}[k_{\lambda}*(v-v_0)]\,dxdt &\geq \int^{t_1}_{0}\int_{\Omega} \partial_{t}[k_{\lambda}*(H_{\epsilon}(v^+)-H_{\epsilon}(v^+_0)]\,dxdt\\
     &=\int_{\Omega}k_{\lambda}*(H_{\epsilon}(v^+)-H_{\epsilon}(v^+_0))(t_1,x)\,dx
\end{split}
\end{equation*}
for $t_1\in (0,T)$. Thus, we infer that
\begin{equation*}
\int_{\Omega}k_{\lambda}*(H_{\epsilon}(v^+)-H_{\epsilon}(v^+_0))(t_1,x)\,dx \leq \int^{t_1}_{0}\int_{\Omega}F^+\,dxdt+ \int^{t_1}_{0}\int_{\Omega}H'_{\epsilon}(v^+) \partial_{t}[(k_{\lambda}-k)*(v-v_0)]\,dxdt.
\end{equation*}
According to the fact that $k_{\lambda} \to k $ in $L^1(0,T)$ as $\lambda\to 0$, it follows from the Young's inequality that
\begin{equation*}
    \int_{\Omega}k_{\lambda}*(H_{\epsilon}(v^+)-H_{\epsilon}(v^+_0))(t_1,x)\,dx \to \int_{\Omega}k*(H_{\epsilon}(u^+)-H_{\epsilon}(v^+_0))(t_1,x)\,dx
\end{equation*}
in $L^1(0,T)$, and selecting a subsequence if necessary, a.e. on $(0,T)$. In addition, we notice that $k*(v-v_0)\in {}_{0}W^{1,2}(0,T;(X_{0}^{s,2}(\Omega))')$. Thus, we deduce that
\begin{equation*}
    \partial_{t}(k_{\lambda}*(v-v_0))\to  \partial_{t}(k*(v-v_0))\quad \text{in}~L^2(0,T;  (X_{0}^{s,2}(\Omega))'),
\end{equation*}
which implies
\[\lim_{\lambda\to 0}\int^{t_1}_{0}\int_{\Omega}H'_{\epsilon}(v^+) \partial_{t}[(k_{\lambda}-k)*(v-v_0)]\,dxdt=0.\]
Then, we conclude that
\begin{equation}\label{bijiao}
   \int_{\Omega}k*(H_{\epsilon}(v^+)-H_{\epsilon}(v^+_0))(t_1
   ,x)\,dx\leq\int^{t_1}_{0}\int_{\Omega}F^+\,dxdt,~\text{for a.e.}~t_1\in (0,T).
\end{equation}
Convolving \eqref{bijiao} with the kernel $l$ and evaluating at $t_1=T$ gives
\begin{equation}
    \int^{T}_{0}\int_{\Omega}H_{\epsilon}(v^+)\,dxdt \leq T\int_{\Omega}H_{\epsilon}(v^+_0)\,dx+ \|l\|_{L^1(0,T)}\int^{T}_{0}\int_{\Omega}F^+\,dxdt.
\end{equation}
Since $H_{\epsilon}(y^+)\to y^+$ as $\epsilon\to 0$ and $|H'_{\epsilon}(y^+)|\leq 1$, it follows from the Lebesgue's dominated convergence theorem that
\begin{equation*}
     \int^{T}_{0}\int_{\Omega}v^+\,dxdt \leq T\int_{\Omega}v^+_0\,dx+ \|l\|_{L^1(0,T)}\int^{T}_{0}\int_{\Omega}F^+\,dxdt.
\end{equation*}
This yields the assertion.
\end{proof}

\section{The proof of main results}
In this section we provide proofs of the main goals--existence and uniqueness of entropy solutions for problems \eqref{eq:main}. Some of the reasoning is based on the ideas developed in \cite{schmitz2023entropy,scholtes2018existence,Teng2019renormalized,petra2021bounded,Rico2009weak,zhang2010renormalized}.  According to  Lemma~\ref{lemma:appro} and Lemma~\ref{lemmabijiao}, we can find  a unique non-negative weak solution $u^m\in
W(u_0, X^{s,2}_{0}(\Omega),L^2(\Omega))$ for the approximate problem~\eqref{eq:mainappro}.   Our aim is to prove that a subsequence of these  solutions $\{u^{m}\}$  converges to a measurable function $u$, which is an entropy solution of problem \eqref{eq:main}. The uniqueness of entropy solution is obtained by a comparison principle. Although some of the arguments are not new, we present a self-contained proof for the sake of clarity and readability. We will divide the proof into several steps.

\medskip
\noindent \textit{Proof of Theorem~\ref{theorem:main}.}
\textbf{Step 1.} Prove the convergence of $u^{m}$ in $L^1(\Omega_T)$ and find its subsequence which is almost everywhere convergent in $\Omega_T$.

It follows from Lemma~\ref{lemmabijiao} that
\begin{equation*}
    \begin{split}
        \sup_{m\in \mathbb{N}}\|u^m\|_{L^1(\Omega_T)} &\leq \sup_{m\in\mathbb{N}}(T\|u^m_{0}\|_{L^1(\Omega)}+\|l\|_{L^1(0,T)}\|f^m\|_{L^1(\Omega_T)})\\
        &\leq T\|u_{0}\|_{L^1(\Omega)}+\|l\|_{L^1(0,T)}\|f\|_{L^1(\Omega_T)}.
    \end{split}
\end{equation*}
Moreover, according to \eqref{eqbijiao1}, we know that
\[u^m\leq u^{m+1} \quad \text{for all}~m\in\N,\]
which implies that  $u^m$ is a non-negative increasing sequence. Therefore, we know that
there exists an element $u$ such that
\[ u^{m} \to u\quad\text{a.e. in}~\Omega_T~\text{for}~m\to\infty.\]
In particular, we have $u^m\leq u$ a.e. in $\Omega_T$. Using the Lebesgue's dominated theorem, we deduce that
\begin{equation*}
     u^{m} \to u\quad\text{in}~L^1(\Omega_T)~\text{for}~m\to\infty.
\end{equation*}

\textbf{Step 2.} Prove $T_{K}(u^{m})$ strongly converges to $T_K(u)$ in $L^2(0,T; X^{s,2}_{0}(\Omega))$ for every $K>0$.

Choosing $T_K(u^{m})$ as a test function in \eqref{eqweaksolution}, we have
\begin{eqnarray*}
    &&\int^{t_1}_{0}\int_{\Omega}T_K(u^{m})\partial_t[k*(u^{m}-u^{m}_0)]\,dxdt\\
    &&\quad +\frac{1}{2}\int^{t_1}_{0}\int_{\mathcal{D}_{\Omega}}(u^{m}(x,t)-u^{m}(y,t))\bigg(T_{K}(u^{m}(x,t))-T_{K}(u^{m}(y,t))\bigg)\,dvdt\\
    &&\quad=\int^{t_1}_{0}\int_{\Omega}f^{m}T_K(u^{m})\,dxdt,
\end{eqnarray*}
for all $t_1\in (0,T)$.

Due to  the  following elementary algebraic inequality
\begin{eqnarray*}
    (a-b)(T_K(a)-T_K(b))\geq |T_K(a)-T_K(b)|^2, \quad a,b\in \R,
\end{eqnarray*}
we deduce that
\begin{eqnarray}
    &&\int^{t_1}_{0}\int_{\Omega}T_K(u^{m})\partial_t[k_{\lambda}*(u^{m}-u^{m}_0)]\,dxdt\\
    &&\quad +\frac{1}{2}\int^{t_1}_{0}\int_{\mathcal{D}_{\Omega}}\bigg|T_{K}(u^{m}(x,t))-T_{K}(u^{m}(y,t))\bigg|^2\,dvdt\nonumber\\
    &&\quad\leq \int^{t_1}_{0}\int_{\Omega}f^{m}T_K(u^{m})\,dxdt+\int^{t_1}_{0}\int_{\Omega}T_K(u^{m})\partial_t[(k_{\lambda}-k)*(u^{m}-u^{m}_0)]\,dxdt.\nonumber
\end{eqnarray}
Combining with Lemma~\ref{lem:equality}, we obtain
\begin{equation}\label{eqstep2j}
\begin{split}
&\int_{\Omega}\left[k_\lambda *\int^{u^{m}}_0 T_K(\sigma)\,d\sigma\right](t_1)\,dx+\int^{t_1}_{0}\int_{\Omega}\bigg[T_{K}(u^m(t))u^m-\int^{u^m}_{0}T_{K}(\sigma)\,d\sigma\bigg]k_{\lambda}(t)\,dxdt\\
&\quad+\int^{t_1}_{0}\int_{\Omega}\int^{t}_{0}\bigg[\int^{u^m(t-s)}_{u^m(t)}T_{K}(\sigma)\,d\sigma-T_{K}(u^m(t))(u^m(t-s)-u^m(t))\bigg][-k'_{\lambda}(s)]\,dsdxdt\\
&\quad +\frac{1}{2}\int^{t_1}_{0}\int_{\mathcal{D}_{\Omega}}\bigg|T_{K}(u^{m}(x,t))-T_{K}(u^{m}(y,t))\bigg|^2\,dvdt\\
    &\quad\leq \int^{t_1}_{0}\int_{\Omega}f^{m}T_K(u^{m})\,dxdt+\int^{t_1}_{0}\int_{\Omega}T_K(u^{m})\partial_t[(k_{\lambda}-k)*(u^{m}-u^{m}_0)]\,dxdt\\
    &\qquad+\int^{t_1}_{0}\int_{\Omega}T_K(u^{m})k_{\lambda}(t)u^{m}_{0}\,dxdt.
    \end{split}
\end{equation}
Since $(u^{m}-u^{m}_0)\in D(B_1)$ the second term in the right-hand side converges to zero as $\lambda\to 0$. Recalling the fact that $k_{\lambda}\to k$ in $L^{1}(0,T)$ as $\lambda \to 0$,  we obtain
\begin{eqnarray*}
    &&\int_{\Omega}\left[k*\int^{u^{m}}_0 T_K(\sigma)\,d\sigma\right](t_1)\,dx +\frac{1}{2}\int^{t_1}_{0}\int_{\mathcal{D}_{\Omega}}\bigg|T_{K}(u^{m}(x,t))-T_{K}(u^{m}(y,t))\bigg|^2\,dvdt\nonumber\\
    &&\quad \leq \int^{t_1}_{0}\int_{\Omega}f^{m}T_K(u^{m})\,dxdt+\int^{t_1}_{0}\int_{\Omega}T_K(u^{m})k(t)u^{m}_{0}\,dxdt,\nonumber
\end{eqnarray*}
which implies that
\begin{eqnarray*}
&&\frac{1}{2}\int^{t_1}_{0}\int_{\mathcal{D}_{\Omega}} \bigg|T_{K}(u^{m}(x,t))-T_{K}(u^{m}(y,t))\bigg|^2\,dvdt\\
&&\quad\leq\int^{t_1}_{0}\int_{\Omega}f^{m}T_K(u^{m})\,dxdt+K\int^{t_1}_{0}\int_{\Omega}k(t)u^{m}_{0}\,dxdt\nonumber  \\
  &&\quad\leq K\|f\|_{L^1(\Omega_T)}+K\|k\|_{L^1(0,T)}\|u_{0}\|_{L^1(\Omega)}\leq C,
\end{eqnarray*}
since $k$ is non-negative. Then, up to a subsequence, we deduce that
\begin{eqnarray*}
 T_K(u^{m})&\overset{m\to\infty}\rightharpoonup T_K(u) \quad \text{weakly in} ~L^2(0,T;X^{s,2}_{0}(\Omega)).
\end{eqnarray*}

Next, we aim to prove  that
\begin{equation*}
    \begin{split}
&\limsup_{m\to\infty}\int^{t_1}_{0}\int_{\mathcal{D}_{\Omega}}\bigg|T_{K}(u^{m}(x))-T_{K}(u^{m}(y))\bigg|^2\,dvdt\\        &\quad\leq\int^{t_1}_{0}\int_{\mathcal{D}_{\Omega}}\bigg|T_{K}(u(x))-T_{K}(u (y))\bigg|^2\,dvdt.
    \end{split}
\end{equation*}
For $l>0$ we denote by $h_l$ the function defined by
\[h_l(u)=\min\{(l+1-|u|)^+,1\}.\]
We take  $\bigg(T_{K}(u^{m})-h_l(u^{m})(T_K(u))_{\mu}\bigg)$ as a test function in equation \eqref{eqweaksolution}, where $(T_{K}(u))_{\mu}$ is defined by \eqref{eqmu}. Thus, we obtain
\begin{equation}\label{eq:step2main}
    J^{\lambda,m,\mu,l}_{1}+ J^{\lambda,m,\mu,l}_{2}= J^{\lambda,m,\mu,l}_{3}+ J^{\lambda,m,\mu,l}_{4},
\end{equation}
where
\begin{equation*}
    \begin{split}
&J^{\lambda,m,\mu,l}_{1}=\int^{t_1}_{0}\int_{\Omega}\bigg(T_{K}(u^{m})-h_l(u^{m})(T_K(u))_{\mu}\bigg)\partial_t[k_{\lambda}*(u^{m}-u^{m}_0)]\,dxdt,\\
    &J^{\lambda,m,\mu,l}_{2}=\int^{t_1}_{0}\langle (-\Delta)^{s}_{2}u^m, T_{K}(u^{m})-h_l(u^{m})(T_K(u))_{\mu}\rangle\, dt,\\
&J^{\lambda,m,\mu,l}_{3}=\int^{t_1}_{0}\int_{\Omega}f^{m}\bigg(T_{K}(u^{m})-h_l(u^{m})(T_K(u))_{\mu}\bigg)\,dxdt,\\
    &J^{\lambda,m,\mu,l}_{4}=\int^{t_1}_{0}\int_{\Omega}(T_{K}(u^{m})-h_l(u^{m})(T_K(u))_{\mu})\partial_t[(k_{\lambda}-k)*(u^{m}-u^{m}_0)]\,dxdt.
    \end{split}
\end{equation*}

We are going to pass the limit with $\lambda\to 0$, $m\to\infty$, then $\mu\to 0$, and finally with $l\to\infty$. Roughly speaking, we show that the limits of $J^{\lambda,m,\mu,l}_{3}$ and $J^{\lambda,m,\mu,l}_{4}$ are zero, and the limit of $J^{\lambda,m,\mu,l}_{1}$ is non-negative. Then the limit of $J^{\lambda,m,\mu,l}_{2}$ is nonpositive.

\textbf{Limit of $J^{\lambda,m,\mu,l}_{3}$.}  To deal with the limit with $m\to\infty$, we apply the Lebesgue's dominated convergence theorem due to the continuity of the integrand and the fact that $u^{m}\to u$ a.e. in $\Omega_{T}$. Moreover, we know that  $(T_{K}(u))_{\mu}$ strongly convergence to $T_{K}(u)$  in $L^2(0,T; X_{0}^{s,2}(\Omega))$ and a.e. in $\Omega_T$ as $\mu\to \infty$, and $h_{l}(u)\to 1$ a.e. in $\Omega_T$ as $l \to \infty$. Therefore, it is obvious that
\begin{equation}
    \begin{split}
     \lim_{l\to \infty}\lim_{\mu\to 0}\lim_{m\to\infty}\int^{t_1}_{0}\int_{\Omega}f^{m}\bigg(T_{K}(u^{m})-h_l(u^{m})(T_K(u))_{\mu}\bigg)\,dxdt=0.
    \end{split}
\end{equation}

\textbf{Limit of $J^{\lambda,m,\mu,l}_{4}$.} As an immediate consequence of $(u^{m}-u^{m}_0)\in D(B_1)$, we obtain that
\begin{equation}
    \begin{split}
       \lim_{\lambda\to 0} \int^{t_1}_{0}\int_{\Omega}(T_{K}(u^{m})-h_l(u^{m})(T_K(u))_{\mu})\partial_t[(k_{\lambda}-k)*(u^{m}-u^{m}_0)]\,dxdt=0.
    \end{split}
\end{equation}

\textbf{Limit of $J^{\lambda,m,\mu,l}_{1}$.}  We are aim to prove that
\begin{equation}\label{eq:budengshi}
\begin{split}
      &\liminf_{l\to \infty}\liminf_{\mu\to 0}\liminf_{m\to\infty}\liminf_{\lambda\to 0}\int^{t_1}_{0}\int_{\Omega}\partial_{t}[k_{\lambda}*(u^{m}-u^{m}_{0})]\\
&\qquad\qquad\qquad\times \bigg(T_{K}(u^{m})-h_l(u^{m})(T_K(u))_{\mu}\bigg)\,dxdt\geq 0
\end{split}
\end{equation}
for almost every $t_1 \in (0,T)$. Let us consider a decomposition
\begin{equation}\label{eqJI}
    \begin{split}
        &\int^{t_1}_{0}\int_{\Omega}\partial_{t}[k_{\lambda}*(u^{m}-u^{m}_{0})]\times \bigg(T_{K}(u^{m})-h_l(u^{m})(T_K(u))_{\mu}\bigg)\,dxdt\\
        &\quad=\int^{t_1}_{0}\int_{\Omega}\partial_{t}[k_{\lambda}*u^{m}]\times T_{K}(u^{m})\,dxdt-\int^{t_1}_{0}\int_{\Omega}\partial_{t}[k_{\lambda}*u^{m}]\times  h_l(u^{m})(T_K(u))_{\mu}\,dxdt\\
        &\qquad-\int^{t_1}_{0}\int_{\Omega}k_{\lambda}u^{m}_{0}\times \bigg(T_{K}(u^{m})-h_l(u^{m})(T_K(u))_{\mu}\bigg)\,dxdt\\
        &\quad=: J^{\lambda,m,\mu,l}_{11}+ J^{\lambda,m,\mu,l}_{12}+J^{\lambda,m,\mu,l}_{13}.
    \end{split}
\end{equation}
Next, we shall give the estimates of  $J^{\lambda,m,\mu,l}_{1i} (i=1,2,3)$ one by one.

\emph{Limit of $J^{\lambda,m,\mu,l}_{13}$.}
As $k_\lambda\xrightarrow{\lambda\to0} k $  in $L^1(0,T)$ and  $(T_{K}(u^{m})-h_l(u^{m})(T_K(u))_{\mu})$ is uniformly bounded by $2K$, combining with the fact that $u^{m}\xrightarrow{m\to\infty} u$ a.e. in $\Omega_T$, we conclude from Lebesgue's dominated  theorem that
\begin{eqnarray*}
&& \int^{t_1}_{0}\int_{\Omega}k_{\lambda}u^{m}_{0}[T_{K}(u^{m})-h_l(u^{m})(T_K(u))_{\mu}]\,dxdt \\ && \quad \xrightarrow[m\to \infty]{\lambda\to 0}\int^{t_1}_{0}\int_{\Omega}k(t)u_0 [T_{K}(u)-h_l(u)(T_K(u))_{\mu}]\,dxdt.
\end{eqnarray*}
Since $(T_K(u))_{\mu}\xrightarrow{\mu\to 0} T_K(u)$ a.e in $\Omega_T$, and $h_l(u)\xrightarrow{l\to\infty} 1 $ for every $u\in \mathbb{R}$,  using the  Lebesgue's dominated  theorem again, we have
\begin{eqnarray*}
    \lim_{l\to\infty}\lim_{\mu\to 0}\int^{t_1}_{0}\int_{\Omega}k(t)u_0 [T_{K}(u)-h_l(u)(T_K(u))_{\mu}]\,dxdt=0.
\end{eqnarray*}
Thus, we conclude that
\begin{equation}
    \begin{split}
     &\liminf_{l\to \infty}\liminf_{\mu\to 0}\liminf_{m\to\infty}\liminf_{\lambda\to 0} J^{\lambda,m,\mu,l}_{13}\\
      &\quad=\liminf_{l\to \infty}\liminf_{\mu\to 0}\liminf_{m\to\infty}\liminf_{\lambda\to 0} -\int^{t_1}_{0}\int_{\Omega}k_{\lambda}u^{m}_{0}\times \bigg(T_{K}(u^{m})-h_l(u^{m})(T_K(u))_{\mu}\bigg)\,dxdt\\
      &\quad=0.
    \end{split}
\end{equation}

\emph{Limit of $J^{\lambda,m,\mu,l}_{11}$.}  According to Lemma~\ref{lem:equality}, we obtain
\begin{equation*}\label{eqdelete313}
    \begin{split}
        \int^{t_1}_{0}\int_{\Omega}\partial_{t}(k_{\lambda}*u^{m})T_{K}(u^{m})\,dxdt&=\int_{\Omega}\bigg[k_{\lambda}*\int^{u^{m}}_{0}T_K(\sigma)\,d\sigma \bigg](t_1)\,dx\\
         &\qquad+\int^{t_1}_{0}\int_{\Omega}\bigg[T_{K}(u^{m}(t))u^{m}-\int^{u^{m}}_{0}T_{K}(\sigma)\,d\sigma\bigg]k_{\lambda}(t)\,dxdt\\
&\qquad+\int^{t_1}_{0}\int_{\Omega}\int^{t}_{0}\bigg[\int^{u^{m}(t-s)}_{u^{m}(t)}T_{K}(\sigma)\,d\sigma\\
        &\qquad\qquad-T_{K}(u^{m}(t))(u^{m}(t-s)-u^{m}(t))\bigg][k'_{\lambda}(s)]\,dsdxdt.
\end{split}
\end{equation*}
Since $u^{m}\to u$ a.e. in $\Omega_T$, it follows that $\int^{u^{m}}_{0}T_{K}(\sigma)\,d\sigma \to \int^{u}_{0}T_{K}(\sigma)\,d\sigma$ a.e. in $\Omega_T$. Combining with the fact that  $\left|\int^{u^{m}}_{0}T_{K}(\sigma)\,d\sigma\right|\leq K|u^{m}|\leq K|u|$,  we  deduce  by the  Lebesgue's dominated theorem that \begin{equation}\label{eqdelete38}
    \int^{u^{m}}_{0}T_{K}(\sigma)\,d\sigma \xrightarrow{m\to\infty} \int^{u}_{0}T_{K}(\sigma)\,d\sigma \quad \text{in }L^1(\Omega_T)
\end{equation}
and
\begin{equation*}
    \int_{\Omega}\int^{u^{m}}_{0}T_{K}(\sigma)\,d\sigma \,dx\xrightarrow{m\to\infty} \int_{\Omega}\int^{u}_{0}T_{K}(\sigma)\,d\sigma dx\quad \text{in } L^1(0,T).
\end{equation*}
Then, we can find an a.e. convergent subsequence, still denoted the same way, such that
\[ \int_{\Omega}\int^{u^{m}}_{0}T_{K}(\sigma)\,d\sigma dx  \xrightarrow{m\to\infty} \int_{\Omega}\int^{u}_{0}T_{K}(\sigma)\,d\sigma dx\quad \text{a.e. in } (0,T).\]
As $k_{\lambda}\to k$ in $L^{1}(0,T)$,  we  conclude by Young's inequality that
\begin{equation*}
    \int_{\Omega}\left[k_{\lambda}*\int^{u^{m}}_{0}T_{K}(\sigma)\,d\sigma\right](\cdot)\, dx\xrightarrow[m\to \infty]{\lambda\to 0} \int_{\Omega}\left[k*\int^{u}_{0}T_{K}(\sigma)\,d\sigma\right](\cdot)\,dx
\end{equation*}
in $L^1(0,T)$ and a.e. in $(0,T)$ for  subsequence.

According to
\begin{equation*}
    \partial_{t}k_{\lambda}*(u^m-u^m_0) \to   \partial_{t}k*(u^m-u^m_0) \quad\text{in } L^2(0,T;  (X_{0}^{s,2}(\Omega))')
\end{equation*}
as $\lambda\to 0$,  we  deduce  that
\begin{equation*}
\lim_{\lambda\to\infty}\int^{t_1}_{0}\int_{\Omega}T_{K}(u^m)\partial_{t}(k_{\lambda}-k)*(u^m-u^m_0) \,dxdt=0,
\end{equation*}
which implies that there exists a constant $C$ which is independent of $\lambda$ such that
\begin{equation*}
   \int^{t_1}_{0}\int_{\Omega}T_{K}(u^m)\partial_{t}(k_{\lambda}-k)*(u^m-u^m_0) \,dxdt \leq C.
\end{equation*}
Thus,  combining with the fact that the kernel $k$ satisfies $\rm(K1)$ and $\rm(K2)$, we deduce from \eqref{eqstep2j}  that for all $t_1\in (0,T)$,
\begin{eqnarray*}
    &&\int^{t_1}_{0}\int_{\Omega}\bigg[T_{K}(u^{m}(t))u^{m}-\int^{u^{m}}_{0}T_{K}(\sigma)\,d\sigma\bigg]k_{\lambda}(t)\,dxdt\\
&&\quad+\int^{t_1}_{0}\int_{\Omega}\int^{t}_{0}\bigg[\int^{u^{m}(t-s)}_{u^{m}(t)}T_{K}(\sigma)\,d\sigma-T_{K}(u^{m}(t))(u^{m}(t-s)-u^{m}(t))\bigg][-k'_{\lambda}(s)]\,dsdxdt\\
&&\quad\leq C,
\end{eqnarray*}
where  $C=C(K,m)$  is independent of $\lambda$. Then, the  Fatou's lemma yields that
\begin{equation}\label{eqdelete316}
    [T_{K}(u^m)u^m-\int^{u^m}_{0}T_{K}(\sigma)\,d\sigma]k(t)\in L^1(\Omega_T),
\end{equation}
and
\begin{equation}\label{eqdelete317}
   \int^{t}_{0}\bigg[\int^{u^m(t-s)}_{u^m(t)}T_{K}(\sigma)\,d\sigma-T_{K}(u^m(t))(u^m(t-s)-u^m(t))\bigg][-k'(s)]\,ds\in L^1(\Omega_T).
\end{equation}
Combining with \eqref{eqstep2j}, we know that
\begin{equation*}
    \begin{split}
       &\liminf_{m\to\infty}\int^{t_1}_{0}\int_{\Omega}\bigg[T_{K}(u^{m}(t))u^{m}-\int^{u^{m}}_{0}T_{K}(\sigma)\,d\sigma\bigg]k(t)\,dxdt\\
&\quad+\liminf_{m\to\infty}\int^{t_1}_{0}\int_{\Omega}\int^{t}_{0}\bigg[\int^{u^{m}(t-s)}_{u^{m}(t)}T_{K}(\sigma)\,d\sigma-T_{K}(u^{m}(t))(u^{m}(t-s)-u^{m}(t))\bigg][-k'(s)]\,dsdxdt \\
&\quad\leq \limsup_{m\to\infty}\int^{t_1}_{0}\int_{\Omega}f^{m}T_K(u^{m})\,dxdt+\limsup_{m\to\infty}\int^{t_1}_{0}\int_{\Omega}T_K(u^{m})k(t)u^{m}_{0}\,dxdt\\
&\quad \leq K\|f\|_{L^1(\Omega_T)}+K\|k(t)\|_{L^1(0,T)}\|u_0\|_{L^1 (\Omega)}.
\end{split}
\end{equation*}
Therefore,  using  the Fatou's lemma again we conclude that
\begin{equation}\label{eqdelete318}
    \begin{split}
     &\liminf_{l\to \infty}\liminf_{\mu\to 0}\liminf_{m\to\infty}\liminf_{\lambda\to 0} J^{\lambda,m,\mu,l}_{11}\\
      &\quad=\liminf_{l\to \infty}\liminf_{\mu\to 0}\liminf_{m\to\infty}\liminf_{\lambda\to 0} \int^{t_1}_{0}\int_{\Omega}\partial_{t}(k_{\lambda}*u^{m})T_{K}(u^{m})\,dxdt\\
      &\quad\geq\int_{\Omega}[k*\int^{u}_{0}T_K(\sigma)\,d\sigma ](t_1)\,dx+\int^{t_1}_{0}\int_{\Omega}\bigg[T_{K}(u(t))u-\int^{u}_{0}T_{K}(\sigma)\,d\sigma\bigg]k(t)\,dxdt\\
    &\qquad+\int^{t_1}_{0}\int_{\Omega}\int^{t}_{0}\bigg[\int^{u(t-s)}_{u(t)}T_{K}(\sigma)\,d\sigma\\
        &\qquad\qquad\qquad\qquad-T_{K}(u(t))(u(t-s)-u(t))\bigg][-k'(s)]\,dsdxdt.
    \end{split}
\end{equation}

\emph{Limit of $J^{\lambda,m,\mu,l}_{12}$.}
Thanks to Lemma \ref{lemconvex}, we get
\begin{equation}\label{step2I}
   \begin{split}
    J^{\lambda,m,\mu,l}_{12}&=-\int^{t_1}_{0}\int_{\Omega}\partial_{t}(k_{\lambda}*u^{m})h_{l}(u^{m})(T_K(u^{m}))_{\mu}\,dxdt\\
    &=-\int^{t_1}_{0}\int_{\Omega}\partial_{t}(k_{\lambda}*\int^{u^{m}}_{0}h_{l}(\sigma)\,d\sigma) (T_K(u^{m}))_{\mu}\,dxdt\\
    &\quad-\int^{t_1}_{0}\int_{\Omega}[h_{l}(u^{m})u^{m}-\int^{u^{m}}_{0}h_{l}(\sigma)\,d\sigma]k_{\lambda}(t)(T_K(u^{m}))_{\mu}\,dxdt\\
    &\quad-\int^{t_1}_{0}\int_{\Omega}\int^{t}_0\bigg[\int^{u^{m}(t-s)}_{u^{m}(t)}h_{l}(\sigma)d\sigma\\
    &\qquad\qquad\qquad-h_{l}(u^{m}(t))(u^{m}(t-s)-u^{m}(t))\bigg][-k'_{\lambda}(s)]\,ds (T_K(u^{m}))_{\mu}\,dxdt\\
    &:=-I^1_{\lambda,m,\mu,l}-I^2_{\lambda,m,\mu,l}-I^3_{\lambda,m,\mu,l}.
\end{split}
\end{equation}
First, we are going to prove that
\begin{equation*}
    \limsup_{l\to \infty}\limsup_{\mu\to 0}\limsup_{m\to\infty}\limsup_{\lambda\to 0}(I^2_{\lambda,m,\mu,l}+I^3_{\lambda,m,\mu,l})=0.
\end{equation*}
Define $T^{+}_{l,l+1}(r):=T_{l,l+1}(\max(r,0))$ and $T^{-}_{l,l+1}(r):=T_{l,l+1}(-\max(-r,0))$ for $l\in \mathbb{R}$. Since $h_l(\sigma)=T^{-}_{l,l+1}(\sigma)-T^{+}_{l,l+1}(\sigma)+1$, it follows
\begin{equation*}
\begin{split}
I^2_{\lambda,m,\mu,l}&=\int^{t_1}_{0}\int_{\Omega}[T^{-}_{l,l+1}(u^{m})u^{m}-\int^{u^{m}}_{0}T^{-}_{l,l+1}(\sigma)d\sigma]k_{\lambda}(t)(T_K(u^{m}))_{\mu}\,dxdt\\
      &\quad-\int^{t_1}_{0}\int_{\Omega}[T^{+}_{l,l+1}(u^{m})u^{m}-\int^{u^{m}}_{0}T^{+}_{l,l+1}(\sigma)d\sigma]k_{\lambda}(t)(T_K(u^{m}))_{\mu}\,dxdt
\end{split}
\end{equation*}
and
\begin{align*}
I^3_{\lambda,m,\mu,l}&=\int^{t_1}_{0}\int_{\Omega}\int^{t}_0\bigg[\int^{u^{m}(t-s)}_{u^{m}(t)}T^{-}_{l,l+1}(\sigma) d\sigma\\
 &\qquad\qquad\qquad -T^{-}_{l,l+1}(u^{m})(u^{m}(t-s)-u^{m}(t))\bigg][-k'_{\lambda}(s)]\,ds (T_K(u^{m}))_{\mu}\,dxdt\\
     &\quad-\int^{t_1}_{0}\int_{\Omega}\int^{t}_0\bigg[\int^{u^{m}(t-s)}_{u^{m}(t)}T^{+}_{l,l+1}(\sigma) d\sigma\\
     &\qquad\qquad\qquad -T^{+}_{l,l+1}(u^{m})(u^{m}(t-s)-u^{m}(t))\bigg][-k'_{\lambda}(s)]\,ds(T_K(u^{m}))_{\mu}\,dxdt.
\end{align*}
Choosing $T_{l,l+1}(u^{m})=T_{1}(u^{m}-T_{l}(u^{m}))$ ~($l>0$)
as a test function for problem \eqref{eq:mainappro}, we obtain
\begin{equation}\label{eq_lmn}
\begin{split}
&\frac{1}{2}\int^{t_1}_{0}\int_{\{(u(x,t),u(y,t))\in R_{l}\}}|u^{m}(x,t)-u^{m}(y,t)|\,dvdt+\int_{\Omega}[k_{\lambda}*\int^{u^{m}}_{0}T_{K}(\sigma)\,d\sigma](t_1)\,dx\\
&\quad+\int^{t_1}_{0}\int_{\Omega}[T_{l,l+1}(u^{m})u^{m}-\int^{u^{m}}_{0}T_{l,l+1}(\sigma)d\sigma]k_{\lambda}(t)\,dxdt\\
&\quad+\int^{t_1}_{0}\int_{\Omega}\int^{t}_{0}\bigg[\int^{u^{m}(t-s)}_{u^{m}(t)}T_{l,l+1}(\sigma)d\sigma\\
&\qquad\qquad\qquad\qquad-T_{l,l+1}(u^{m})(u^{m}(t-s)-u^{m}(t))\bigg][-k'_{\lambda}(s)]\,ds\,dxdt\\
&\quad\leq \int_{\Omega_{T} \cap\{|u^{m}|>l\}}f^{m}\,dxdt+\int_{\Omega_{T}\cap\{|u^{m}|>l\}}k_{\lambda}u^{m}_{0}\,dxdt\\
&\quad \quad+\int_{\Omega_{T}\cap\{|u^{m}|>l\}}T_{l,l+1}(u^{m})|\partial_{t}[(k_{\lambda}-k)*(u^{m}-u^{m}_{0})]|\,dxdt
\end{split}
\end{equation}
for every $t_1\in (0,T)$, where
\[R_{l}=\{(w,v)\in \mathbb{R}^2: l+1 \leq \max{\{|w|,|v|\}}~and ~ \min{\{|w|,|v|\}}\leq l~or~wv<0\}.\]
Since $k_{\lambda}\to k$ in $L^1(0,T)$, $u^{m}\to u$ a.e. in $\Omega_T$  and $|u|<\infty$ a.e. in $\Omega_T$ it follows that
\[ \limsup_{l\to\infty}\limsup_{m\to \infty}\limsup_{\lambda\to 0}\bigg(\int_{\Omega_{T}\cap\{|u^{m}|>l\}}f^{m}\,dxdt+\int_{\Omega_{T}\cap\{|u^{m}|>l\}}k_{\lambda}u^{m}_{0}\,dxdt\bigg)=0.\]
Moreover, since $(u_m-u_0)\in D(B_1)$, one immediately obtain that the last term in \eqref{eq_lmn} tends to 0 as $\lambda \to 0$.
Since all terms on the left-hand side in \eqref{eq_lmn} are non-negative, we
conclude that, for every $t_1\in(0,T)$,
\begin{equation}
\begin{split}
&\limsup_{l\to\infty}\limsup_{m\to \infty}\limsup_{\lambda\to 0}\int^{t_1}_{0}\int_{\Omega}[T_{l,l+1}(u^{m})u^{m}-\int^{u^{m}}_{0}T_{l,l+1}
(\sigma)d\sigma]k_{\lambda}\,dxdt=0,\\
&\limsup_{l\to\infty}\limsup_{m\to \infty}\limsup_{\lambda\to 0}\int^{t_1}_{0}\int_{\Omega}\int^{t}_{0}\bigg[\int^{u^{m}(t-s)}_{u^{m}(t)}T_{l,l+1}(\sigma)d\sigma\\
&\qquad\qquad\qquad\qquad\qquad\qquad-T_{l,l+1}(u^{m}(t))\bigg(u^{m}(t-s)-u^{m}\bigg)\bigg][-k'_{\lambda}(s)]\,dsdxdt=0.
\end{split}
\end{equation}
It is easy to see that the corresponding result holds in the case $T_{l,l+1}$ is replaced by $T^{\pm}_{l,l+1}$. Thus, we obtain
\[\limsup_{l\to \infty}\limsup_{\mu\to 0}\limsup_{m\to\infty}\limsup_{\lambda\to 0}(I^2_{\lambda,m,\mu,l}+I^3_{\lambda,m,\mu,l})=0.\]

Next, we will focus on  the term $I^1_{\lambda,m,\mu,l}$. It is easy to check that $\int^{u^{m}}_{0} h_{l}(\sigma)\,d\sigma$ is uniformly bounded.
Thus, it follows from the Lebesgue's dominated theorem that
\begin{eqnarray*}
    \int^{u^{m}}_{0} h_{l}(\sigma)\,d\sigma \xrightarrow{m\to\infty} \int^{u}_{0}h_{l}(\sigma)\,d\sigma\quad \text{in } L^2(0,T; L^2(\Omega)).
\end{eqnarray*}
Moreover, we know that
\begin{eqnarray*}
   B_{\lambda} \int^{u^{m}}_{0}h_{l}(\sigma)\,d\sigma=\partial_t(k_{\lambda}*\int^{u^{m}}_{0}h_{l}(\sigma)\,d\sigma)\in L^2(0,T;  (X_{0}^{s,2}(\Omega))').
\end{eqnarray*}
Using the fact that
\[J^B_{\mu}B_{\lambda}v=\frac{B}{(I+\mu B)(I+\lambda B)}v=B_{\mu}J^B_{\lambda}v \quad \text{for any } v\in D(B),\]
we obtain
\begin{equation}\label{eqdelete319}
\begin{split}
I^{1}_{\lambda,m,\mu,l}&=\int^{t_1}_{0}\int_{\Omega}\partial_{t}(k_{\lambda}*\int^{u^{m}}_{0}h_{l}(\sigma)\,d\sigma) (T_K(u))_{\mu}\,dxdt\\
      &=\bigg\langle B_{\lambda}\int^{u^{m}}_{0}h_{l}(\sigma)\,d\sigma,J^{B^*}_{\mu}T_{K}(u)\bigg\rangle_{L^2(0,t_1; (X_{0}^{s,2}(\Omega))') \times L^2(0,t_1; X_{0}^{s,2}(\Omega)) }\\
      &=\bigg\langle B_{\mu}J^{B}_{\lambda} \int^{u^{m}}_{0}h_{l}(\sigma)\,d\sigma, T_{K}(u)\bigg \rangle_{L^2(0,t_1; (X_{0}^{s,2}(\Omega))') \times L^2(0,t_1; X_{0}^{s,2}(\Omega))}
\end{split}
\end{equation}
for any $t_1\in (0,T)$. Since $J^{B}_{\lambda}$ is bounded  operator satisfying $J^{B}_{\lambda}v \to v$ in  $L^2(0,T;  (X_{0}^{s,2}(\Omega))')$ for all $v\in L^2(0,T;  (X_{0}^{s,2}(\Omega))')$ as $\lambda \to 0$, the continuous embedding $L^2(\Omega)\hookrightarrow (X_{0}^{s,2}(\Omega))' $ yields
\begin{eqnarray*}
    J_{\lambda}^{B} \int^{u^{m}}_{0} h_{l}(\sigma)\,d\sigma \xrightarrow{\lambda\to 0} \int^{u^{m}}_{0}h_{l}(\sigma)\,d\sigma \xrightarrow{m\to\infty} \int^{u}_{0}h_{l}(\sigma)\,d\sigma\quad \text{in}~~L^2(0,T;  (X_{0}^{s,2}(\Omega))').
\end{eqnarray*}
Thus, the continuity of the Yosida approximation $B_{\mu}$ implies
\begin{equation*}
    \begin{split}
        I^{1}_{\lambda,m,\mu,l} &\xrightarrow[m\to\infty]{\lambda\to 0} \langle B_{\mu} \int^{u}_{0}h_{l}(\sigma)\,d\sigma, T_{K}(u) \rangle\\
        &=\int^{t_1}_{0}\int_{\Omega}\partial_t(k_{\mu}*\int^{u}_{0} h_{l}(\sigma)\,d\sigma)T_K(u)\,dxdt.
    \end{split}
\end{equation*}
Define $g_{l}(u):= \int^{u}_{0}h_{l}(\sigma)\,d\sigma$, $u\in \mathbb{R}$ and
\[H(u):=\int^{u}_{0}T_{K}\circ g^{-}_{l}(\sigma)\,d\sigma,\quad u\in ran(g_{l}).\]
Using the transformation $\sigma\mapsto g_{l}(\sigma)$, we see that
\[H(g_{l}(u(t)))=\int^{u(t)}_{0}T_{K}(\sigma)\,dg_{l}(\sigma).\]
By the definition of $g_l$, the measure $dg_l$ is absolutely continuous with respect to $d\sigma$ and the Radon-Nikodym derivative is given by $h_l$. According to Lemma~\ref{lem:equality}, we get
\begin{equation}\label{eqdelete320}
    \begin{split}
        &-\int^{t_1}_{0}\int_{\Omega}\partial_t\bigg(k_{\mu}*\int^{u}_{0} h_{l}(\sigma)\,d\sigma\bigg)T_K(u)\,dxdt\\
        &=-\int_{\Omega}\bigg[k_{\mu}*\int^{u}_{0}T_K(\sigma)h_l(\sigma)\,d\sigma \bigg](t_1)\,dx\\
         &\quad-\int^{t_1}_{0}\int_{\Omega}\bigg[T_{K}(u(t))\int^{u(t)}_{0}h_{l}(\sigma)\,d\sigma-\int^{u(t)}_{0}T_{K}(\sigma)h_l(\sigma)\,d\sigma\bigg]k_{\mu}(t)\,dxdt\\
        &\quad-\int^{t_1}_{0}\int_{\Omega}\int^{t}_{0}\bigg[\int^{u(t-s)}_{u(t)}T_{K}(\sigma)h_l(\sigma)\,d\sigma-T_{K}(u(t))\int^{u(t-s)}_{u^(t)}h_{l}(\sigma)\,d\sigma\bigg][-k'_{\mu}(s)]\,dsdxdt.
    \end{split}
\end{equation}

According to the  Young's inequality and  $k_{\mu}\to k$ in $L^1(0,T)$, we have
\[\int_{\Omega}\bigg[k_{\mu}*\int^{u}_{0}T_K(\sigma)h_l(\sigma)\,d\sigma \bigg]dx \xrightarrow{\mu\to 0} \int_{\Omega}\bigg[k*\int^{u}_{0}T_K(\sigma)h_l(\sigma)\,d\sigma \bigg]dx\]
in $L^1(0,T)$ and a.e. in $(0,T)$ for  subsequence. In addition, we note that
\[\int^{u}_{0}T_K(\sigma)h_l(\sigma)\,d\sigma\leq \int^{u}_{0}T_K(\sigma)\,d\sigma\quad\text{a.e. in }\Omega_{T},\]
then the Lebesgue's dominated theorem yields  that
\[\int_{\Omega}\bigg[k*\int^{u}_{0}T_K(\sigma)h_l(\sigma)\,d\sigma \bigg]\,dx\xrightarrow{l\to\infty} \int_{\Omega}\bigg[k*\int^{u}_{0}T_K(\sigma)\,d\sigma \bigg]\,dx\quad \text{in } L^1(0,T)\]
 and  a.e. in $(0,T)$ for  subsequence.

Note that $0\leq h_{l}\leq 1$ also implies
\begin{equation*}
    \begin{split}
        0&\leq T_{K}(u(t))\int^{u}_{0}h_{l}(\sigma)\,d\sigma-\int^{u}_{0}T_{K}(\sigma)h_l(\sigma)\,d\sigma\\
        &\leq T_{K}(u(t))u(t)-\int^{u}_{0}T_{K}(\sigma)\,d\sigma
    \end{split}
\end{equation*}
a.e. on $\Omega_T$, and
\begin{equation*}
    \begin{split}
       0&\leq \int^{u(t-s)}_{u(t)}T_{K}(\sigma)h_l(\sigma)\,d\sigma-T_{K}(u(t))\int^{u(t-s)}_{u(t)}h_{l}(\sigma)\,d\sigma\\
       &\leq \int^{u(t-s)}_{u(t)}T_{K}(\sigma)\,d\sigma-T_{K}(u(t))(u(t-s)-u(t))
\end{split}
\end{equation*}
a.e. on $\Omega_{T}$.
Due to  $k_{\mu}$ is non-negative, non-increasing and satisfies (K1) and (K2), we deduce from the Lebesgue's dominated theorem that
\begin{equation*}
    \begin{split}
        &\liminf_{l\to\infty}\liminf_{\mu\to 0}\int^{t_1}_{0}\int_{\Omega}\bigg[T_{K}(u(t))\int^{u}_{0}h_{l}(\sigma)\,d\sigma-\int^{u}_{0}T_{K}(\sigma)h_l(\sigma)\,d\sigma\bigg]k_{\mu}(t)\,dxdt\\
         &\quad=\int^{t_1}_{0}\int_{\Omega}\bigg[T_{K}(u(t))u(t)-\int^{u}_{0}T_{K}(\sigma)\,d\sigma\bigg]k(t)\,dxdt
    \end{split}
\end{equation*}
and
\begin{equation*}
    \begin{split}
&\liminf_{l\to\infty}\liminf_{\mu\to 0} \int^{t_1}_{0}\int_{\Omega}\int^{t}_{0}\bigg[\int^{u(t-s)}_{u(t)}T_{K}(\sigma)h_l(\sigma)\,d\sigma-T_{K}(u(t))\int^{u(t-s)}_{u^(t)}h_{l}(\sigma)\,d\sigma\bigg][-k'_{\mu}(s)]\,dsdxdt\\
&\quad= \int^{t_1}_{0}\int_{\Omega}\int^{t}_{0}\bigg[\int^{u(t-s)}_{u(t)}T_{K}(\sigma)\,d\sigma-T_{K}(u(t))(u(t-s)-u(t))\bigg][-k'(s)]\,dsdxdt.
    \end{split}
\end{equation*}
Thus, passing to the limit in \eqref{eqdelete319} yields
\begin{equation*}\label{eqdelete321}
    \begin{split}
     &\liminf_{l\to \infty}\liminf_{\mu\to 0}\liminf_{m\to\infty}\liminf_{\lambda\to 0} J_{12}^{\lambda,m,\mu,l} \\
     &\quad=-\liminf_{l\to \infty}\liminf_{\mu\to 0}\liminf_{m\to\infty}\liminf_{\lambda\to 0} I^{1}_{\lambda,m,\mu,l}\\
        &\quad=-\int_{\Omega}[k*\int^{u}_{0}T_K(\sigma)\,d\sigma ](t_1)\,dx\\
         &\quad\quad-\int^{t_1}_{0}\int_{\Omega}\bigg[T_{K}(u(t))u-\int^{u}_{0}T_{K}(\sigma)\,d\sigma\bigg]k(t)\,dxdt\\
        &\quad\quad-\int^{t_1}_{0}\int_{\Omega}\int^{t}_{0}\bigg[\int^{u(t-s)}_{u(t)}T_{K}(\sigma)\,d\sigma-T_{K}(u(t))(u(t-s)-u(t))\bigg][-k'(s)]\,dsdxdt.
    \end{split}
\end{equation*}
Combining with \eqref{eqJI} and \eqref{eqdelete318}, we obtain
\begin{equation*}
    \begin{split}
    &\liminf_{l\to \infty}\liminf_{\mu\to 0}\liminf_{m\to\infty}\liminf_{\lambda\to 0}\int^{t_1}_{0}\int_{\Omega}J_{1}^{\lambda,m,n,\mu,l}\\
    &\quad=\liminf_{l\to \infty}\liminf_{\mu\to 0}\liminf_{m\to\infty}\liminf_{\lambda\to 0}\int^{t_1}_{0}\int_{\Omega}J_{11}^{\lambda,m,n,\mu,l}+J_{12}^{\lambda,m,n,\mu,l}\geq 0,
    \end{split}
\end{equation*}
which end the proof of \eqref{eq:budengshi}.

\textbf{Limit of $J^{\lambda,m,\mu,l}_{2}$ }
Therefore, we conclude that for $t_1\in (0,T)$,
\begin{equation}
    \begin{split}
        &\limsup_{l\to \infty}\limsup_{\mu\to 0}\limsup_{m\to\infty}\limsup_{\lambda\to 0}  J^{\lambda,m,\mu,l}_{2}\\
        &\quad =\limsup_{l\to \infty}\limsup_{\mu\to 0}\limsup_{m\to\infty}\int^{t_1}_{0}\langle (-\Delta)^{s}_{2}u^m, T_{K}(u^{m})-h_l(u^{m})(T_K(u))_{\mu}\rangle dt\\
        &\quad\leq 0.
        \end{split}
\end{equation}
Then, it is easy to deduce that
\begin{equation*}
\begin{split}
&\limsup_{m\to\infty}\int^{T}_{0}\langle (-\Delta)^{s}_{2}u^m, T_{K}(u^{m})-T_K(u)\rangle dt\\
        &\quad\leq\limsup_{l\to \infty}\limsup_{\mu\to 0}\limsup_{m\to\infty}\int^{t_1}_{0}\langle (-\Delta)^{s}_{2}u^m, h_{l}(u^m)(T_{K}(u))_{\mu}-T_K(u)\rangle dt.
\end{split}
\end{equation*}
 Denote that
\begin{equation*}
    \begin{split}
     & H^{m,\mu,l}:=(u^m(x)-u^m(y))\bigg(h_{l}(u^m(x))(T_{K}(u(x)))_{\mu}-T_{K}(u(x))\\
         &\qquad\qquad\qquad\qquad\qquad\qquad\qquad-(h_{l}(u^m(y))(T_{K}(u(y)))_{\mu}-T_{K}(u(y)))\bigg).
    \end{split}
\end{equation*}
We have
  \begin{equation*}
     \begin{split}
         &\int^{T}_{0}\langle (-\Delta)^{s}_{2}u^m, h_{l}(u^m)(T_{K}(u))_{\mu}-T_K(u)\rangle dt=\int^{T}_{0}\int_{\mathcal{D}_{\Omega}}H^{m,\mu,l}\,dvdt.
     \end{split}
 \end{equation*}
Taking $l\geq K$, we set
\begin{equation*}
    \begin{split}
        & D_1=\{(x,y,t)\in \mathcal{D}\times (0,T): u^m(x,t)\leq l,\qquad u^m(y,t)\leq l\},\\
        & D_2=\{(x,y,t)\in \mathcal{D}\times (0,T): u^m(x,t)\leq l, \qquad l+1\leq u^m(y,t)\},\\
        & D_3=\{(x,y,t)\in \mathcal{D}\times (0,T): u^m(x,t)\leq l, \qquad l\leq u^m(y,t)\leq l+1\},\\
           & D_4=\{(x,y,t)\in \mathcal{D}\times (0,T): l\leq u^m(x,t)\leq l+1, \qquad u^m(y,t)\leq l\},\\
             & D_5=\{(x,y,t)\in \mathcal{D}\times (0,T): l \leq u^m(x,t)\leq l+1, \qquad
 l+1\leq u^m(y,t)\},\\
  & D_6=\{(x,y,t)\in \mathcal{D}\times (0,T): l\leq u^m(x,t)\leq l+1,\qquad l \leq u^m(y,t)\leq l+1\},\\
              & D_7=\{(x,y,t)\in \mathcal{D}\times (0,T): l+1\leq u^m(x,t),\qquad u^m(y,t)\leq l\},\\
              & D_8=\{(x,y,t)\in \mathcal{D}\times (0,T):l+1\leq u^m(x,t),\qquad l+1\leq u^m(y,t)\},\\
     & D_9=\{(x,y,t)\in \mathcal{D}\times (0,T): l+1 \leq u^m(x,t),\qquad l\leq u^m(y,t)\leq l+1\}.
    \end{split}
\end{equation*}
Then
\[\mathcal{D}\times (0,T)=D_1\cup D_2\cup D_3\cup D_4\cup D_5\cup D_6\cup D_7\cup D_8\cup D_9.\]

In $D_1$, we have $h_{l}(u^m(x))=h_{l}(u^m(y))=1$. Therefore, we obtain
\begin{equation*}
    \begin{split}
&\limsup_{\mu\to 0}\limsup_{m\to\infty} \int_{D_1} H^{m,\mu,l}\,dvdt\\
&\quad =\limsup_{\mu\to 0}\limsup_{m\to\infty}\int_{D_1}(u^m(x)-u^m(y))\bigg((T_{K}(u(x)))_{\mu}-T_{K}(u(x))\\
&\qquad\qquad\qquad\qquad-((T_{K}(u(y)))_{\mu}-T_{K}(u(y)))\bigg)\,dvdt\\
&\quad \leq \limsup_{\mu\to 0}\limsup_{m\to\infty}\int_{D}|T_{l}(u^m(x))-T_{l}(u^{m}(y))|\bigg|(T_{K}(u(x)))_{\mu}-T_{K}(u(x))\\
&\qquad\qquad\qquad\qquad-((T_{K}(u(y)))_{\mu}-T_{K}(u(y)))\bigg|\,dvdt.
    \end{split}
\end{equation*}
Since $(T_{K}(u))_{\mu}$ strongly convergence to $T_{K}(u)$  in $L^2(0,T; X_{0}^{s,2}(\Omega))$ and a.e. in $\Omega_T$ as $\mu \to 0$, it follows that
\begin{equation*}
\limsup_{\mu\to 0}\limsup_{m\to\infty} \int_{D_1} H^{m,\mu,l}\,dvdt  \leq 0.
\end{equation*}

In $D_2$, we have
\begin{equation*}
    \begin{split}
\int_{D_2} H^{m,\mu,l}\,dvdt=\int_{D_2}(u^m(x)-u^m(y))\bigg((T_{K}(u(x)))_{\mu}-T_{K}(u(x))+T_{K}(u(y))\bigg)\,dvdt.
    \end{split}
\end{equation*}
Combining with the fact that $K\leq l+1\leq u^m(y)\leq u(y)$, we know that
\[T_{K}(u(y))-T_{K}(u(x)) \geq 0.\]
Thus, we conclude that
\begin{equation*}
\limsup_{\mu\to 0}\limsup_{m\to\infty} \int_{D_2} H^{m,\mu,l}\,dvdt  \leq 0.
\end{equation*}
We point out out that the estimates in $D_7$ are similarly to the case in $D_2$.

In $D_3$, we have
\begin{equation*}
    \begin{split}
&\limsup_{\mu\to 0}\limsup_{m\to\infty} \int_{D_3} H^{m,\mu,l}\,dvdt\\
&\quad =\limsup_{\mu\to 0}\limsup_{m\to\infty}\int_{D_3}(u^m(x)-u^m(y))\bigg((T_{K}(u(x)))_{\mu}-T_{K}(u(x))\\
&\qquad\qquad\qquad\qquad-(h_{l}(u^m(y))(T_{K}(u(y)))_{\mu}-T_{K}(u(y)))\bigg)\,dvdt\\
&\quad \leq \limsup_{\mu\to 0}\limsup_{m\to\infty}\int_{D}|T_{l+1}(u^m(x))-T_{l+1}(u^{m}(y))|\bigg|(T_{K}(u(x)))_{\mu}-T_{K}(u(x))\\
&\qquad\qquad\qquad\qquad-((T_{K}(u(y)))_{\mu}-T_{K}(u(y)))\bigg|\,dvdt.
    \end{split}
\end{equation*}
Similarly to the estimates in $D_1$, we have
\begin{equation*}
\limsup_{\mu\to 0}\limsup_{m\to\infty} \int_{D_3} H^{m,\mu,l}\,dvdt  \leq 0.
\end{equation*}

In $D_4$,  it can be done similarly to the estimate in $D_3$.

In $D_5$,  we have
\begin{equation*}
    \begin{split}
\int_{D_5} H^{m,\mu,l}\,dvdt=\int_{D_5}(u^m(x)-u^m(y))\bigg(h_{l}(u^m(x))(T_{K}(u(x)))_{\mu}-T_{K}(u(x))+T_{K}(u(y))\bigg)\,dvdt.
    \end{split}
\end{equation*}
Similarly to estimate in $D_2$, we have
\begin{equation*}
\limsup_{\mu\to 0}\limsup_{m\to\infty} \int_{D_5} H^{m,\mu,l}\,dvdt  \leq 0.
\end{equation*}
It is easy to check that the estimates in $D_9$ are similarly to the case in $D_5$.

In $D_6$, since $K\leq l \leq u^m(x)\leq u(x)$, $K\leq l \leq u^m(y)\leq u(y)$, we have
\begin{equation*}
    \begin{split}
\int_{D_6} H^{m,\mu,l}\,dvdt &= \int_{D_6} (u^m(x)-u^m(y))\bigg(h_{l}(u^m(x))-h_{l}(u^m(y))\bigg)K_{\mu}\,dvdt,\\
&\leq K\int_{D_6} (u^m(x)-u^m(y))^2 h'_{l}(\xi)\,dvdt\\
&\leq 0,
    \end{split}
\end{equation*}
where $\xi$  is between $u^m(x)$ and $u^m(y)$.

In $D_8$, since $h_{l}(u^m(x))=h_{l}(u^m(y))=0$, $T_{K}(u(x))=T_{K}(u(y))=K$, we immediately obtain
\begin{equation*}
   \int_{D_9}H^{m,\mu,l}\,dvdt=0.
\end{equation*}

Combining all the above estimates, we get
\begin{equation*}
\begin{split}
&\limsup_{m\to\infty}\int^{T}_{0}\langle (-\Delta)^{s}u^m, T_{K}(u^{m})-T_K(u)\rangle dt\\
        &\quad\leq\limsup_{l\to \infty}\limsup_{\mu\to 0}\limsup_{m\to\infty}\int^{T}_{0}\langle (-\Delta)^{s}u^m, h_{l}(u^m)(T_{K}(u))_{\mu}-T_K(u)\rangle dt\\
        &\quad\leq 0,
\end{split}
\end{equation*}
which implies that
\begin{equation}
    \begin{split}
       &\limsup_{m\to\infty}\int^{T}_{0}\int_{\mathcal{D}_{\Omega}}(u^{m}(x)-u^{m}(y))(T_{K}(u^{m}(x))-T_{K}(u^{m}(y)))\,dvdt\\
        &\quad\leq \limsup_{m\to\infty} \int^{T}_{0}\int_{\mathcal{D}_{\Omega}}(u^{m}(x)-u^{m}(y))(T_{K}(u(x))-T_{K}(u (y)))\,dvdt.
    \end{split}
\end{equation}
Then following the arguments of the proof of \cite [Lemma~3.6]{abdellaoui2019on}, we have
\begin{equation}
    \begin{split}
&\limsup_{m\to\infty}\int^{T}_{0}\int_{\mathcal{D}_{\Omega}}\bigg|T_{K}(u^{m}(x))-T_{K}(u^{m}(y))\bigg|^2\,dvdt\\
            &\quad\leq\int^{T}_{0}\int_{\mathcal{D}_{\Omega}}\bigg|T_{K}(u(x))-T_{K}(u (y))\bigg|^2\,dvdt.
    \end{split}
\end{equation}
Since $T_{K}(u^m)\overset{m\to\infty}\rightharpoonup T_K(u)$ weakly in $L^2(0,T;X^{s,2}_{0}(\Omega))$, we obtain $T_{K}(u^m)\xrightarrow{m\to\infty} T_{K}(u)$ strongly in  $L^2(0,T;X^{s,2}_{0}(\Omega))$.

\textbf{Step 3.} Show that $u$ is an entropy solution.

Let $S\in\mathcal{P}$, $\phi\in X^{s,2}_{0}(\Omega)\cap L^{\infty}(\Omega)$, $\zeta\in C^{\infty}_{0}([0,T))$, $\zeta\geq 0$. Taking $S(u^m-\phi)\zeta$ as a test function in \eqref{eqweaksolution}, we obtain
\begin{equation*}
\begin{split}
   &\int^{T}_{0}\int_{\Omega}\zeta S(u^m-\phi)\partial_{t}[k*(u^m-u^m_0)]\,dxdt\\
   &\quad+\frac{1}{2}\int^{T}_{0}\int_{\mathcal{D}_{\Omega}}\zeta(u^m(x)-u^m(y))(S(u^m-\phi)(x)-S(u^m-\phi)(y))\,dvdt\\
   &\quad=\int_{0}^{T}\int_{\Omega} f^m S(u^m-\phi)\zeta\,dxdt.
\end{split}
\end{equation*}
Choosing arbitrary $k_1,k_2\in L^1(0,T)$ non-increasing and non-negative with $k_2(0^+)<\infty$ such that $k=k_1+k_2$.   For $\lambda>0$ we define $k_{1,\lambda}$ by the kernel associated to the Yosida-approximation of the operator $B^1:=\partial_{t}(k_1*\cdot)$, $D(B^1):=\{\omega\in L^2(0,T;(X^{s,2}_{0}(\Omega))');k_1*\omega \in {}_{0}W^{1,2}(0,T; (X^{s,2}_{0}(\Omega))')\}$. Then there holds
\begin{equation*}
\begin{split}
   &\int^{T}_{0}\int_{\Omega}\zeta S(u^m-\phi)\partial_{t}[k_{1,\lambda}*(u^m-u^m_0)]\,dxdt+\int^{T}_{0}\int_{\Omega}\zeta S(u^m-\phi)\partial_{t}[k_2*(u^m-u^m_0)]\,dxdt\\
   &\quad+\frac{1}{2}\int^{T}_{0}\int_{\mathcal{D}_{\Omega}}\zeta(u^m(x)-u^m(y))(S(u^m-\phi)(x,t)-S(u^m-\phi)(y,t))\,dvdt\\
   &\quad=\int_{0}^{T}\int_{\Omega} f^m S(u^m-\phi)\zeta\,dxdt+\int^{T}_{0}\int_{\Omega}\zeta S(u^m-\phi)\partial_{t}[(k_{1,\lambda}-k_1)*(u^m-u^m_0)]\,dxdt.
\end{split}
\end{equation*}
According to Lemma~\ref{lemH}, we obtain
\begin{equation*}
\begin{split}
   &-\int^{T}_{0}\int_{\Omega}\zeta_t (k_{1,\lambda}*\int^{u^m}_{u^m_0}S(\sigma-\phi)\,d\sigma dxdt
   +\int^{T}_{0}\int_{\Omega}\zeta S(u^m-\phi)\partial_{t}[k_2*(u^m-u^m_0)]\,dxdt\\
   &\quad+\frac{1}{2}\int^{T}_{0}\int_{\mathcal{D}_{\Omega}}\zeta(u^m(x)-u^m(y))(S(u^m-\phi)(x)-S(u^m-\phi)(y))\,dvdt\\
   &\quad=\int_{0}^{T}\int_{\Omega} f^m S(u^m-\phi)\zeta\,dxdt+\int^{T}_{0}\int_{\Omega}\zeta S(u^m-\phi)\partial_{t}[(k_{1,\lambda}-k_1)*(u^m-u^m_0)]\,dxdt.\end{split}
\end{equation*}
Recalling the fact that $k_{1,\lambda}\to k_{1}$ in $L^1(0,T)$ and $u^m-u^{m}_{0}\in D(B^1)$, we obtain $\partial_{t}(k_{1,\lambda}*(u^m-u^m_{0}))\to \partial_{t}(k_{1}*(u^m-u^m_{0})) $ in $L^2(0,T;(X^{s,2}_{0}(\Omega))')$. Passing to the limit as $\lambda\to 0$ in the above equation, we have
\begin{equation}\label{eq:entropy}
    \begin{split}
&-\int^{T}_{0}\int_{\Omega}\zeta_t\bigg(k_1*\int^{u^m}_{u^m_0}S(\sigma-\phi)\,d\sigma\bigg)\,dxdt+\int^{T}_{0}\int_{\Omega}\zeta\partial_t[k_2*(u^m-u^m_0)]S(u^m-\phi)\,dxdt\\
        &\quad +\frac{1}{2}\int^{T}_{0}\int_{\mathcal{D}_{\Omega}} (u^m(x)-u^m(y))\bigg(S(u^m(x,t)-\phi(x,t))-S(u^m(y,t)-\phi(y,t))\bigg)\zeta\,dvdt \\
        &\leq \int^{T}_{0}\int_{\Omega}\zeta f^mS(u^m-\phi)\,dxdt.
    \end{split}
\end{equation}

 We  first consider the term in the right-hand side of \eqref{eq:entropy}. Since $S$ is bounded and continuous, we immediately see that
\begin{equation*}
\lim_{m\to\infty}\int^{T}_{0}\int_{\Omega}\zeta f^mS(u^m-\phi)\,dxdt= \int^{T}_{0}\int_{\Omega}\zeta fS(u-\phi)\,dxdt.
\end{equation*}
 We  now deal with the first term  of \eqref{eq:entropy}. Since $u^m\xrightarrow[]{m\to\infty} u$ a.e. in $\Omega_T$, it follows that
  \[\int^{u^m}_{u^m_0}S(\sigma-\phi)\,d\sigma \to \int^u_{u_0}S(\sigma-\phi)\,d\sigma\quad \text{a.e in } \Omega_T.\]
Since $|u^m|\leq |u|$, it follows that
\begin{equation*}
    \bigg|\int^{u^m}_{u^m_0}S(\sigma-\phi)\,d\sigma\bigg| \leq C|u|\in L^1(\Omega_T).
\end{equation*}
Hence, we deduced from the Lebesgue's dominated theorem that
\begin{equation*}
\lim_{m\to\infty}\int^{T}_{0}\int_{\Omega}\zeta_t\bigg(k_1*\int^{u^m}_{u^m_0}S(\sigma-\phi)\,d\sigma\bigg)\,dxdt=\int^{T}_{0}\int_{\Omega}\zeta_t\bigg(k_1*\int^{u}_{u_0}S(\sigma-\phi)\,d\sigma\bigg)\,dxdt.
\end{equation*}
We now deal with the second term of \eqref{eq:entropy}. Note that
    \begin{equation*}
         \begin{split}
             &\partial_{t}[k_2*(u^m-u^m_{0})](t)\\
             &\quad=k_2(0^+)(u^m(t)-u^m_0)+\int^t_{0}u^m(t-s)-u^m_{0}\,dk_2(s) \\
             &\overset{m\to\infty}{\to}k_2(0^+)(u(t)-u_0)+\int^t_{0}u(t-s)-u_{0}\,dk_2(s)
         \end{split}
    \end{equation*}
 in $L^1(\Omega_T)$. Consequently, we have
 \begin{equation*}
     \lim_{m\to\infty}\int^{T}_{0}\int_{\Omega}\zeta\partial_t[k_2*(u^m-u^m_0)]S(u^m-\phi)\,dxdt=\int^{T}_{0}\int_{\Omega}\zeta\partial_t[k_2*(u-u_0)]S(u-\phi)\,dxdt.
 \end{equation*}
Next, we focus on the third term in the left-hand side of equation~\eqref{eq:entropy}. We set
\[w^m=u^m-\phi,\quad U^m(x,y,t)=u^m(x,t)-u^m(y,t) \quad \text{and}\quad W^m(x,y,t)=w^m(x,t)-w^m(y,t),\]
then
\begin{equation*}
    \begin{split}
       &U^m(x,y,t)[S(w^m(x,t))-S(w^m(y,t))]\zeta\\
       &\quad=:M^{1,m}(x,y,t)\zeta+M^{2,m}(x,y,t)\zeta,
    \end{split}
\end{equation*}
where
\[M^{1,m}(x,y,t)=W^m(x,y,t)[S(w^m(x,t))-S(w^m(y,t))],\]
and
\[M^{2,m}(x,y,t)=[U^m(x,y,t)-W^m(x,y,t)][S(w^m(x,t))-S(w^m(y,t))].\]
Since $\zeta \in \mathcal{D}([0,T))$, $\zeta\geq 0$,
\begin{equation*}
    \begin{split}
        M^{1,m}(x,y,t)&=W^m(x,y,t)[S(w^m(x,t))-S(w^m(y,t))]\\
        &=S'(\omega_m(\cdot))[w^m(x,t)-w^m(y,t)]^2\geq 0,
    \end{split}
\end{equation*}
and
\begin{equation*}
    M^{1,m}(x,y,t)\xrightarrow[]{m\to\infty} W(x,y,t)[S(w(x,t))-S(w(y,t))]\quad \text{a.e. in } \Omega_T,
\end{equation*}
where $w=u-\phi$ and $W(x,y,t)=w(x,t)-w(y,t)$, it follows from the Fatou's Lemma that
\begin{equation*}
    \begin{split}
\lim_{m\to\infty}\int^{T}_{0}\int_{\mathcal{D}_{\Omega}}  M^{1,m}(x,y,t)\zeta\,dvdt\geq \int^{T}_{0}\int_{\mathcal{D}_{\Omega}}W(x,y,t)[S(w(x,t))-S(w(y,t))]\zeta\,dvdt.
    \end{split}
\end{equation*}
In addition, it is easy to check that
\begin{equation*}
    \begin{split}
|M^{2,m}(x,y,t)| \leq |\phi(x,t)-\phi(y,t)||S(w^m(x,t))-S(w^m(y,t))|.
    \end{split}
\end{equation*}
Since $S\in \mathcal{P}$, we claim that there exists a constant $K$ large enough such that $S(v)=S(T_{K}(v))$ for any $v\in \R$.
Recalling the fact that $T_{K}(u^m)\xrightarrow{m\to \infty}T_{K}(u)$ strongly  in $L^2(0,T;X^{s,2}_{0}(\Omega))$, we get
\[S(w^m)\xrightarrow{m\to\infty} S(w) \quad \text{ strongly in } L^2(0,T;X^{s,2}_{0}(\Omega)),\]
which implies that
\begin{equation}
    \frac{S(w^m(x,t))-S(w^m(y,t))}{|x-y|^{\frac{N+2s}{2}}}\to  \frac{S(w(x,t))-S(w(y,t))}{|x-y|^{\frac{N+2s}{2}}}
\end{equation}
   strongly in $L^2(0,T;\mathcal{D}_{\Omega})$. Since $\phi\in X^{s,2}_{0}(\Omega)\cap L^\infty(\Omega)$, we deduce by duality argument  that
\begin{equation*}
\begin{split}
  &\int^{T}_{0}\int_{\mathcal{D}_{\Omega}}\zeta(\phi(x,t)-\phi(y,t))[S(w^m(x,t))-S(w^m(y,t))]\,dvdt \\
  &\quad\xrightarrow{m\to \infty}\int^{T}_{0}\int_{\mathcal{D}_{\Omega}}\zeta(\phi(x,t)-\phi(y,t))[S((w(x,t))-S(w(y,t))]\,dvdt.
\end{split}
\end{equation*}
Using the Lebesgue's dominated theorem, we get
\begin{equation*}
    \begin{split}
        &\lim_{m\to\infty}\int^{T}_{0}\int_{\mathcal{D}_{\Omega}} \zeta M^{2,m}(x,y,t)\,dvdt \\
        &\quad= \int^{T}_{0}\int_{\mathcal{D}_{\Omega}}\zeta[U(x,y,t)-W(x,y,t)][S(w(x,t))-S(w(y,t))]\,dvdt.
    \end{split}
\end{equation*}
Therefore, we conclude that
\begin{equation*}
\begin{split}
    &\lim_{m\to\infty}\int^{T}_{0}\int_{\mathcal{D}_{\Omega}}\zeta U^m(x,y,t)[S(w^m(x,t))-S(w^m(y,t))]\,dvdt \\
    &\quad \geq \int^{T}_{0}\int_{\mathcal{D}_{\Omega}}\zeta U(x,y,t)[S(w(x,t))-S(w(y,t))]\,dvdt.
\end{split}
\end{equation*}
Combining all the above estimates we conclude that
    \begin{align*}
        &-\int^{T}_{0}\int_{\Omega}\zeta_t\bigg[k_1*\int^{u}_{u_0}S(\sigma-\phi)\,d\sigma\bigg]+\int^{T}_{0}\int_{\Omega}\zeta\partial_t[k_2*(u-u_0)]S(u-\phi)\,dxdt\\
        &\quad +\frac{1}{2}\int^{T}_{0}\int_{\mathcal{D}_{\Omega}}(u(x)-u(y))\bigg(S(u(x,t)-\phi(x,t))-S(u(y,t)-\phi(y,t))\bigg)\zeta\,dvdt\nonumber\\
        &\quad\leq \int^{T}_{0}\int_{\Omega}\zeta fS(u-\phi)\,dxdt\nonumber
    \end{align*}
  for all $\phi\in X^{s,2}_{0}(\Omega)\cap L^{\infty}(\Omega)$, $\zeta\in C^{\infty}_{0}([0,T))$, $\zeta\geq 0$, $S\in \mathcal{P}$, and $k_1,k_2 \in L^1(0,T)$ non-increasing and non-negative with $k=k_1+k_2$ and $k_2(0^+)<\infty$. This completes the proof of Theorem~\ref{theorem:main}. \hfill$\Box$

\medskip

\noindent \textit{Proof of Proposition \ref{prop}}. Suppose that $u^{im}(i=1,2)$ are weak solutions to problem \eqref{eq:mainappro} with $u_0=u^{im}_{0}$, $f=f^{im}$ respectively. We take the test function as $H'_{\epsilon}((u^{1m}-u^{2m})^+)$ in the weak formulation of the approximate problem for both $u^{1m}$ and $u^{2m}$. We point out that the function $H_{\epsilon}$ is defined as \eqref{eqH}. Arguing as the proof of Lemma~\ref{lemmabijiao}, we obtain
 \begin{eqnarray*}
 \int^{T}_{0}\int_{\Omega}(u^{1m}-u^{2m})^{+}\,dxdt \leq T\int_{\Omega}(u_{0}^{1m}-u_{0}^{2m})^{+}\,dx+\|l\|_{L^1(0,T)}\int^{T}_{0}\int_{\Omega}(f^{1m}-f^{2m})^{+}\,dx.
    \end{eqnarray*}
Passing to  the limit with $m\to\infty$, it is easy to deduce that
\begin{eqnarray*}
 \int^{T}_{0}\int_{\Omega}(u^{1}-u^{2})^{+}\,dxdt \leq T\int_{\Omega}(u_{0}^{1}-u_{0}^{2})^{+}\,dx+\|l\|_{L^1(0,T)}\int^{T}_{0}\int_{\Omega}(f^{1}-f^{2})^{+}\,dx.
    \end{eqnarray*}
So, we obtain the desired result. The proof of \eqref{eqprop2} is analogous to  \eqref{eqprop1} with the only difference that in the case of \eqref{eqprop2}, we take the test function as $H'_{\epsilon}(u^{1m}-u^{2m})$.
\hfill$\Box$


\begin{thebibliography}{[a]}
\bibitem{abdellaoui2018on}
B.~Abdellaoui, A.~Attar, R.~Bentifour and I.~Peral.
\newblock On fractional {$p$}-{L}aplacian parabolic problem with general data.
\newblock {\em Ann. Mat. Pura Appl. (4)}, 197(2):329--356, 2018.

\bibitem{abdellaoui2019on}
B.~Abdellaoui, A.~Attar and R.~Bentifour.
\newblock On the fractional {$p$}-{L}aplacian equations with weight and general
  datum.
\newblock {\em Adv. Nonlinear Anal.}, 8(1):144--174, 2019.

\bibitem{nathael2010renormalized}
N.~Alibaud, B.~Andreianov and M. Bendahmane.
\newblock Renormalized solutions of the fractional {L}aplace equation.
\newblock {\em C. R. Math. Acad. Sci. Paris}, 348(13-14):759--762, 2010.

\bibitem{Mark2016a}
M.~Allen, L.~Caffarelli and A.~Vasseur.
\newblock A parabolic problem with a fractional time derivative.
\newblock {\em Arch. Ration. Mech. Anal.}, 221(2):603--630, 2016.

\bibitem{boris2008on}
B.~Andreianov, K.~Sbihi and P.~Wittbold.
\newblock On uniqueness and existence of entropy solutions for a nonlinear
  parabolic problem with absorption.
\newblock {\em J. Evol. Equ.}, 8(3):449--490, 2008.

\bibitem{benilan1995an}
P.~B\'{e}nilan, L.~Boccardo, T.~Gallou\"{e}t, R.~Gariepy, M.~Pierre and J.~Luis V\'{a}zquez.
\newblock An {$L^1$}-theory of existence and uniqueness of solutions of
  nonlinear elliptic equations.
\newblock {\em Ann. Scuola Norm. Sup. Pisa Cl. Sci. (4)}, 22(2):241--273, 1995.

\bibitem{lucio1996existence}
L.~Boccardo, T.~Gallou\"{e}t and L.~Orsina.
\newblock Existence and uniqueness of entropy solutions for nonlinear elliptic
  equations with measure data.
\newblock {\em Ann. Inst. H. Poincar\'{e} C Anal. Non Lin\'{e}aire},
  13(5):539--551, 1996.

\bibitem{luis2012nonlocal}
L.~Caffarelli.
\newblock Non-local diffusions, drifts and games.
\newblock In {\em Nonlinear partial differential equations}, volume~7 of {\em
  Abel Symp.}, pages 37--52. Springer, Heidelberg, 2012.

\bibitem{luis2007an}
L.~Caffarelli and L.~Silvestre.
\newblock An extension problem related to the fractional {L}aplacian.
\newblock {\em Comm. Partial Differential Equations}, 32(7-9):1245--1260, 2007.

\bibitem{luis2011uniform}
L.~Caffarelli and E.~Valdinoci.
\newblock Uniform estimates and limiting arguments for nonlocal minimal
  surfaces.
\newblock {\em Calc. Var. Partial Differential Equations}, 41(1-2):203--240,
  2011.

\bibitem{clement1981asymptotic}
Ph. Cl\'{e}ment and J.~A. Nohel.
\newblock Asymptotic behavior of solutions of nonlinear {V}olterra equations
  with completely positive kernels.
\newblock {\em SIAM J. Math. Anal.}, 12(4):514--535, 1981.

\bibitem{clement1990completely}
P.~Cl\'{e}ment and J.~Pr\"{u}ss.
\newblock Completely positive measures and {F}eller semigroups.
\newblock {\em Math. Ann.}, 287(1):73--105, 1990.

\bibitem{clement1992global}
P.~Cl\'{e}ment and J.~Pr\"{u}ss.
\newblock Global existence for a semilinear parabolic {V}olterra equation.
\newblock {\em Math. Z.}, 209(1):17--26, 1992.

\bibitem{Corta2021}
C.~Cart\'{a}zar, F.~Quir\'{o}s and N.~Wolanski.
\newblock Large-time behavior for a fully nonlocal heat equation.
\newblock{\em Vietnam J. Math.}, 49(3):831--844, 2021.



\bibitem{k2019entropy}
K.~Dareiotis, M.~Gerencs\'{e}r and B.~Gess.
\newblock Entropy solutions for stochastic porous media equations.
\newblock {\em J. Differential Equations}, 266(6):3732--3763, 2019.

\bibitem{D1}
D.~del-Castillo-Negrete, B.A.~Carreras and V.E.~Lynch.
\newblock Fractional diffusion in plasma turbulence.
\newblock {\em Phys. Plasmas}, 11(8): 3854--3864, 2004.

\bibitem{D2}
D.~del-Castillo-Negrete, B.A.~Carreras and V.E.~Lynch.
\newblock Nondiffusive transport in plasma turbulence: A fractional diffusion approach.
\newblock {\em Phys. Rev. Lett.}, 94(6):065003, 2005.

\bibitem{Dier2020on}
D.~Dier, J.Kemppainen, J.~Siljander and R. Zacher.
\newblock On the parabolic Harnack inequality for non-local diffusion equations.
\newblock{\em Math. Z.}, 295(3-4):1751--1769, 2020.

\bibitem{Eleonora2012hit}
E.~Di~Nezza, G. Palatucci and E.~Valdinoci.
\newblock Hitchhiker's guide to the fractional {S}obolev spaces.
\newblock {\em Bull. Sci. Math.}, 136(5):521--573, 2012.

\bibitem{dipierro2019decay}
S.~Dipierro, E.~Valdinoci and V.~Vespri.
\newblock Decay estimates for evolutionary equations with fractional
  time-diffusion.
\newblock {\em J. Evol. Equ.}, 19(2):435--462, 2019.

\bibitem{fu2022global}
Y.~Fu and X.~Zhang.
\newblock Global existence, local existence and blow-up of mild solutions for abstract time-space fractional diffusion equations.
\newblock {\em Topol. Methods Nonlinear Anal.}, 60(2):415--440, 2022.


\bibitem{Gripenberg1990volterra}
G.~Gripenberg, S.-O. Londen and O.~Staffans.
\newblock {\em Volterra integral and functional equations}, volume~34 of {\em
  Encyclopedia of Mathematics and its Applications}.
\newblock Cambridge University Press, Cambridge, 1990.

\bibitem{Volker2004on}
V. G.~Jakubowski and P.~Wittbold.
\newblock On a nonlinear elliptic-parabolic integro-differential equation with
  {$L^1$}-data.
\newblock {\em J. Differential Equations}, 197(2):427--445, 2004.

\bibitem{Kenneth2011a}
K. H.~Karlsen, F.~Petitta and S.~Ulusoy.
\newblock A duality approach to the fractional {L}aplacian with measure data.
\newblock {\em Publ. Mat.}, 55(1):151--161, 2011.

\bibitem{jukka2016decay}
J.~Kemppainen, J.~Siljander, V.~Vergara and R.~Zacher.
\newblock Decay estimates for time-fractional and other non-local in time
  subdiffusion equations in {$\mathbb{R}^d$}.
\newblock {\em Math. Ann.}, 366(3-4):941--979, 2016.

\bibitem{Jukka2017repre}
J.~Kemppainen, J.~Siljander and R.~Zacher.
\newblock Representation of solutions and large-time behavior for fully
  nonlocal diffusion equations.
\newblock {\em J. Differential Equations}, 263(1):149--201, 2017.

\bibitem{kim2016asy}
K-H.~Kim and S.~Lim.
\newblock Asymptotic behaviors of fundamental solution and its derivatives to
  fractional diffusion-wave equations.
\newblock {\em J. Korean Math. Soc.}, 53(4):929--967, 2016.

\bibitem{Janne2019equiv}
J.~Korvenp\"{a}\"{a}, T.~Kuusi and E.~Lindgren.
\newblock Equivalence of solutions to fractional {$p$}-{L}aplace type
  equations.
\newblock {\em J. Math. Pures Appl. (9)}, 132:1--26, 2019.

\bibitem{Tuomo2015nonlocal}
T.~Kuusi, G.~Mingione and Y.~Sire.
\newblock Nonlocal equations with measure data.
\newblock {\em Comm. Math. Phys.}, 337(3):1317--1368, 2015.

\bibitem{landes1989on}
R.~Landes.
\newblock On the existence of weak solutions of perturbated systems with
  critical growth.
\newblock {\em J. Reine Angew. Math.}, 393:21--38, 1989.

\bibitem{Tommaso2015basic}
T.~Leonori, I.~Peral, A.~Primo and F.~Soria.
\newblock Basic estimates for solutions of a class of nonlocal elliptic and
  parabolic equations.
\newblock {\em Discrete Contin. Dyn. Syst.}, 35(12):6031--6068, 2015.

\bibitem{li2021entropy}
Y. Li, F.~Yao and S.~Zhou.
\newblock Entropy and renormalized solutions to the general nonlinear elliptic
  equations in {M}usielak-{O}rlicz spaces.
\newblock {\em Nonlinear Anal. Real World Appl.}, 61:103330, 20,
  2021.

\bibitem{mazon2016frac}
J.M.~Maz\'{o}n, J.D. ~Rossi and J.~Toledo.
\newblock Fractional {$p$}-{L}aplacian evolution equations.
\newblock {\em J. Math. Pures Appl. (9)}, 105(6):810--844, 2016.

\bibitem{Ralf2000the}
R.~Metzler and J.~Klafter.
\newblock The random walk's guide to anomalous diffusion: a fractional dynamics
  approach.
\newblock {\em Phys. Rep.}, 339(1):77, 2000.

\bibitem{Ralf2004the}
R.~Metzler and J.~Klafter.
\newblock The restaurant at the end of the random walk: recent developments in
  the description of anomalous transport by fractional dynamics.
\newblock {\em J. Phys. A}, 37(31):R161--R208, 2004.

\bibitem{Samko2003integro}
S.~Samko and R.P.~Cardoso.
\newblock Integral equations of the first kind of {S}onine type.
\newblock {\em Int. J. Math. Math. Sci.}, (57):3609--3632, 2003.

\bibitem{niklas2020entropy}
N.~Sapountzoglou.
\newblock Entropy solutions to doubly nonlinear integro-differential equations.
\newblock {\em Nonlinear Anal.}, 192:111656, 31, 2020.

\bibitem{schmitz2023entropy}
K.~Schmitz and P.~Wittbold.
\newblock Entropy solutions for time-fractional porous medium type equations.
 arXiv:2302.06399, 2023.

\bibitem{scholtes2018existence}
M.~Scholtes and P.~Wittbold.
\newblock Existence of entropy solutions to a doubly nonlinear
  integro-differential equation.
\newblock {\em Differential Integral Equations}, 31(5-6):465--496, 2018.

\bibitem{Teng2019renormalized}
K.~Teng, C.~Zhang and S.~Zhou.
\newblock Renormalized and entropy solutions for the fractional
  {$p$}-{L}aplacian evolution equations.
\newblock {\em J. Evol. Equ.}, 19(2):559--584, 2019.

\bibitem{Vergara2008lyapunov}
V.~Vergara and R.~Zacher.
\newblock Lyapunov functions and convergence to steady state for differential
  equations of fractional order.
\newblock {\em Math. Z.}, 259(2):287--309, 2008.

\bibitem{vicente2015optimal}
V.~Vergara and R.~Zacher.
\newblock Optimal decay estimates for time-fractional and other nonlocal
  subdiffusion equations via energy methods.
\newblock {\em SIAM J. Math. Anal.}, 47(1):210--239, 2015.

\bibitem{petra2021bounded}
P.~Wittbold, P.~Wolejko and R.~Zacher.
\newblock Bounded weak solutions of time-fractional porous medium type and more
  general nonlinear and degenerate evolutionary integro-differential equations.
\newblock {\em J. Math. Anal. Appl.}, 499(1):125007, 20, 2021.

\bibitem{Rico2008boundedness}
R.~Zacher.
\newblock Boundedness of weak solutions to evolutionary partial
  integro-differential equations with discontinuous coefficients.
\newblock {\em J. Math. Anal. Appl.}, 348(1):137--149, 2008.

\bibitem{Rico2009weak}
R.~Zacher.
\newblock Weak solutions of abstract evolutionary integro-differential
  equations in {H}ilbert spaces.
\newblock {\em Funkcial. Ekvac.}, 52(1):1--18, 2009.

\bibitem{zacher2013aweak}
R.~Zacher.
\newblock A weak {H}arnack inequality for fractional evolution equations with
  discontinuous coefficients.
\newblock {\em Ann. Sc. Norm. Super. Pisa Cl. Sci. (5)}, 12(4):903--940, 2013.

\bibitem{zhang2010entropy}
C.~Zhang and S.~Zhou.
\newblock Entropy and renormalized solutions for the {$p(x)$}-{L}aplacian
  equation with measure data.
\newblock {\em Bull. Aust. Math. Soc.}, 82(3):459--479, 2010.

\bibitem{zhang2010renormalized}
C.~Zhang and S.~Zhou.
\newblock Renormalized and entropy solutions for nonlinear parabolic equations
  with variable exponents and {$L^1$} data.
\newblock {\em J. Differential Equations}, 248(6):1376--1400, 2010.

\end{thebibliography}
\end{document}